\documentclass[10pt]{amsart}
\usepackage{amsmath}
\usepackage{paralist}
\usepackage{graphicx} 
\usepackage{epstopdf}
\usepackage[colorlinks=true]{hyperref}
\usepackage{float}
\usepackage{color,comment}
\usepackage{caption}
\usepackage{wrapfig}
\usepackage{multirow}
\usepackage{tabularx}
\usepackage{booktabs}
\usepackage{subfigure}
\usepackage{chngcntr}
\usepackage{enumitem}
\newtheorem{theorem}{Theorem}[section]

\newtheorem{lemma}[theorem]{Lemma}

\newtheorem{conjecture}{Conjecture}

\newtheorem{example}{Example}[section]
\theoremstyle{definition}

\newtheorem{remark}{Remark}


\begin{document}
\title[LotKa-Volterra competition models]{ Some ``Counterintuitive" Results in Two Species Competition}
\author[ Parshad, Antwi-Fordjour, Takyi]{}
\maketitle

\centerline{\scshape  Rana D. Parshad$^1$,  Kwadwo Antwi-Fordjour $^2$ and Eric M. Takyi$^1$}

\medskip

{\footnotesize
 \centerline{ 1) Department of Mathematics,}
 \centerline{ Iowa State University,}
 \centerline{ Ames, IA 50011, USA}
   }
   \medskip
{\footnotesize
  \centerline{ 2) Department of Mathematics and Computer Science,}
 \centerline{ Samford University,}
 \centerline{ Birmingham, AL 35229, USA}
}
\medskip

\begin{abstract}
We investigate the classical two species ODE and PDE Lotka-Volterra competition models, where one of the competitors could potentially go extinct in finite time. We show that in this setting, classical theories and intuitions do not hold, and various counter intuitive dynamics are possible. In particular, the weaker competitor could \emph{avoid} competitive exclusion, and the slower diffuser may \emph{not} win. Numerical simulations are performed to verify our analytical findings.
\end{abstract}

\section{Introduction}
The two species Lotka-Volterra competition model and its variants have been rigorosly investigated in the last few decades They represent a simplified scenario of two competing species, taking into account growth and inter/intra species competition \cite{Brown1980, Lou2008}. They predict well observed states in ecology, of co-existence, competitive exclusion of one competitor, and bi-stability - and find immense application in applied mathematics, population ecology, invasion science, evolutionary biology and economics, to name a few areas \cite{Lou2008, Cantrell2003, Cantrell2004, CC18}. The equilibrium states are achieved only asymptotically, as is the case in many differential equation population models.
In the current manuscript, we aim to investigate the effect on classical system, when one of the competitors has the potential to go extinct in finite time. There are various motivations to study finite time extinction (FTE) in population dynamics. For example if we are modeling predator-prey densities, then a quantity less than one need not indicate essential extinction - and pest populations could \emph{rebound} from low levels \cite{P19}. This is well observed with soybean aphids (\emph{Aphis glicines}), the chief invasive pest on soybean crop, particularly in the Midwestern US \cite{R11}, that arrival of aphids in very low density ($<<1$) could lead to population levels of several thousand on one leaf, in a matter of 1-2 months \cite{O18}. Another motivation is epidemics, very timely due to the current epidemic because of the COVID19 virus \cite{M20}. Recent work \cite{FCGT18, FT18} has considered a large class of susceptible-infected models with non-smooth incidence functions, that can lead to host extinction in finite time - but are seen to be good fits to modeling disease transmitted by rhanavirus among amphibian populations \cite{G08},  disease transmission in host-parasitoid models \cite{F02}, as well as in virus transmission in gypsy moths \cite{D97}. Non-smooth responses have been considered analytically in the predator-prey literature too \cite{B12, S97, S99}, a careful analysis of this splitting of phase, initial condition dependent extinction, and all of the rich dynamics and bifurcations involved therein, have been considered in \cite{BS19, KRM20}, to the best of our knowledge. They however, have been considered a fair bit in the applied sense due to the good fit they provide to various real data \cite{L12, R05, I08, M04}.

In the current manuscript we show that,

\begin{itemize}

\item FTE in the weaker competitor in the two species ODE Lotka-Volterra competition model, can enable it to avoid competitive exclusion, and persist. This is seen via Lemma \ref{lem:l2}, see Fig. \ref{fig:Extinction-FTE1} (b). FTE in the stronger competitor can lead to bi-stability, via Theorem \ref{thm:FiniteTimeTheorem}, see Fig. \ref{fig:Extinction-FTE1} (c). FTE in the weak competition case, can lead to bi-stability or competitive exclusion, via Theorem \ref{thm:FiniteTimeTheoremWeakCase}, see Fig. \ref{fig:Weak-FTE1}-\ref{fig:Weak-FTE2}. 

\item FTE in the strong competition case, for the spatially homogenous PDE model, can cause diffusion induced recovery - as opposed to diffusion induced extinction seen in the classical case, via Theorem \ref{thm:FTE-ESTIMATE}, see Fig. \ref{fig:estimates}.

\item FTE in the equal kinetics case, in the spatially inhomogenous PDE model, can cause the slower diffuser to \emph{loose}, via Theorem \ref{thm:sdl}, see Fig. \ref{fig:newfte1}.

\item FTE in the weak competition case, in the spatially inhomogenous PDE model, can \emph{change} the bifurcation structure in the space of diffusion parameters, see Fig. \ref{fig:Diff_Diff}.

\end{itemize}

\section{The ODE Case}\label{section:model_formulation}
\subsection{The Extinction/Competitive Exclusion Case} \label{section:Finite_time_extinction}

Consider the classical two species Lotka Volterra competition model,

\begin{equation}\label{GeneralEquation}
\left\{ \begin{array}{ll}
\dfrac{du }{dt} &~ = u (a_{1}-b_{1}u - c_{1}v) ,\\[2ex]
\dfrac{dv }{dt} &~ =  v (a_{2}-b_{2}v - c_{2}u).
\end{array}\right.
\end{equation}
%
%
where $u$ and $v$ are the population densities of two competing species, $a_1$ and $a_2$ are the intrinsic (per capita) growth rates, $b_1$ and $b_2$ are the intraspecific competition rates, $c_1$ and $c_2$ are the interspecific competition rates. All parameters considered are positive constants.

We consider first the competitive exclusion case,

\begin{align}\label{extinctioncondition2}
\dfrac{a_{1}}{a_{2}}>\max\left\lbrace\dfrac{b_{1}}{c_{2}},\dfrac{c_{1}}{b_{2}}\right\rbrace
\end{align}

or

\begin{align}\label{extinctioncondition1}
\dfrac{a_{1}}{a_{2}}<\min\left\lbrace\dfrac{b_{1}}{c_{2}},\dfrac{c_{1}}{b_{2}}\right\rbrace. \quad 
\end{align}

In this setting, as $t\rightarrow \infty$, the solutions $(u(t),v(t))$ converges uniformly to $(a_{1}/b_{1},0)$ or $(0,a_{2}/b_{2})$ irrespective of initial conditions. WLOG we consider the case when $(a_{1}/b_{1},0)$ is globally asymptotically stable, thus $u$ is the stronger competitor and drives $v$ to extinction, and $v$ is said to be competitively excluded \cite{Murray93}. 

We posit that $v$ can \emph{avoid} competitive exclusion by (1) counter intuitively \emph{speeding up} the process to its own demise, via a finite time extinction (FTE) dynamic or also (2) if the stronger competitor $u$ possessed the FTE dynamic.
To this end consider,

\begin{equation}
\label{eq:Ge1}
\left\{ \begin{array}{ll}
\dfrac{du }{dt} &~ = a_{1}u - b_{1}u^{2} - c_{1}u^{p} v ,   \ 0 < p  \leq 1,         \\[2ex]
\dfrac{dv }{dt} &~ =  a_{2} v - b_{2} v^{2} - c_{2} uv^{q}, \ 0 < q \leq 1.
\end{array}\right.
\end{equation}

We see that the classical model is a special case of the above, when $p=q=1$.
Note, $0 < p < 1, q=1$, allows for \emph{finite} time extinction (FTE) of $u$,
 and $0 < q < 1, p=1$, allows for \emph{finite} time extinction (FTE) of $v$.



\begin{lemma}
\label{lem:l1}
Consider \eqref{eq:Ge1}, $q=1 $ and $\eqref{extinctioncondition2}$ holds, then there exists $0 < p < 1$, for which an interior saddle equilibrium occurs.  
\end{lemma}

\begin{proof}
The $u$ and $v$ nullclines are given by,

\begin{equation}
v=f(u)=u^{1-p} \Big(\frac{a_{1}}{c_{1}} - \frac{b_{1}}{c_{1}}u \Big), \ v=g(u)=\frac{a_{2}}{b_{2}} - \frac{c_{2}}{b_{2}}u.
\end{equation}
%
Via $\eqref{extinctioncondition2}$ we must have,
\begin{equation*}
\frac{a_{1}}{c_{1}} - \frac{b_{1}}{c_{1}}u > \frac{a_{2}}{b_{2}} - \frac{c_{2}}{b_{2}}u.
\end{equation*}

Now $f(u)$ is a parabolic shaped polynomial, with zeroes at $u=0, u=\frac{a_{1}}{b_{1}}$. Since the $v$ nullcline is unmoved,  continuity of $f$, and the intermediate value theorem will ensure that there is an intersection of the nullclines in the interior, creating an interior equilibrium. Standard linear analysis proves this is a saddle. 
\end{proof}
For the linear analysis, see Appendix \ref{a1}. Also see Fig. \ref{fig:Extinction-FTE1}.
Now consider the case  where $v$ possesses the FTE dynamic.

\begin{lemma}
\label{lem:l2}
Consider \eqref{eq:Ge1}, $p=1$ and $\eqref{extinctioncondition2}$ holds , then there exists $0< q < 1$, for which two interior  equilibria occur, a saddle and a nodal sink.  
\end{lemma}

\begin{proof}
We consider the nullclines as functions of $v$. The $u$ and $v$ nullclines are given by,

\begin{equation}
u=f_{1}(v)=\frac{a_{1}}{b_{1}} - \frac{c_{1}}{b_{1}}v, \ u=g_{1}(v)= v^{1-q} \left(\frac{a_{2}}{c_{2}} - \frac{b_{2}}{c_{2}}v\right). 
\end{equation}

Again via $\eqref{extinctioncondition2}$ we must have, 
\begin{equation}
 \frac{a_{2}}{c_{2}} - \frac{b_{2}}{c_{2}}v < \frac{a_{1}}{b_{1}} - \frac{c_{1}}{b_{1}}v.
\end{equation}

We proceed by contradiction. Assume there is no intersection of the nullclines for any $v \in \left[0, \frac{a_{1}}{c_{1}}\right]$, and any $0<q<1$. Then we must have that $f_{1}(v) > g_{1}(v)$, for $v \in \left[0, \frac{a_{1}}{c_{1}}\right]$, and any $0<q<1$. This implies,

\begin{equation}
0 < v^{1-q} \left(\frac{a_{2}}{c_{2}} - \frac{b_{2}}{c_{2}}v\right) < \frac{a_{1}}{b_{1}} - \frac{c_{1}}{b_{1}}v, \ v \in \left[0, \frac{a_{1}}{c_{1}}\right], \ \forall \ 0<q<1.
\end{equation}

WLOG let $0< 1- q = \frac{a}{2b} < 1$, thus we must have,

\begin{equation}
\label{eq:eqn1}
0 < v  <\left(  \frac{\left( \frac{a_{1}}{b_{1}} - \frac{c_{1}}{b_{1}}v\right)}{\left(\frac{a_{2}}{c_{2}} - \frac{b_{2}}{c_{2}}v\right)}\right)^{\frac{2b}{a}}
, \ v \in \left[0, \frac{a_{1}}{c_{1}}\right], \ \forall  \ 0<q<1.
\end{equation}

The power of 2 in exponent guarantees positivity, even though for $ \frac{a_{2}}{b_{2}}  < v \leq  \frac{a_{1}}{c_{1}} $, $ \left( \frac{ \left( \frac{a_{1}}{b_{1}} - \frac{c_{1}}{b_{1}}v\right)}{\left(\frac{a_{2}}{c_{2}} - \frac{b_{2}}{c_{2}}v\right)}\right) < 0$. Next we let $v \rightarrow \frac{a_{1}}{c_{1}} $ and so 
\begin{equation}
\left(  \frac{\left(\frac{a_{1}}{b_{1}} - \frac{c_{1}}{b_{1}}v\right)}{\left(\frac{a_{2}}{c_{2}} - \frac{b_{2}}{c_{2}}v\right)}\right)^{\frac{2b}{a}} \rightarrow 0.
\end{equation}

Thus from \eqref{eq:eqn1}, we obtain $0 < \frac{a_{1}}{c_{1}}  < 0$, which is a contradiction. Thus there must exist some $v^{*}\in \left[0, \frac{a_{1}}{c_{1}}\right], \ $ and some $ 0<q<1$, s.t. $f_{1}(v^{*}) < g_{1}(v^{*})$. Now using continuity of $g_{1}$ and the intermediate value theorem, gives us two intersections, thus two equilibria. Standard linearization shows one to be a saddle, the other is seen to be locally stable by standard theory.

\end{proof}
For the linear analysis, see Appendix \ref{a1}. Also see Fig. \ref{fig:Extinction-FTE1}.

%

\begin{figure}[!htb]
\begin{center}
\subfigure[]{
    \includegraphics[scale=.44]{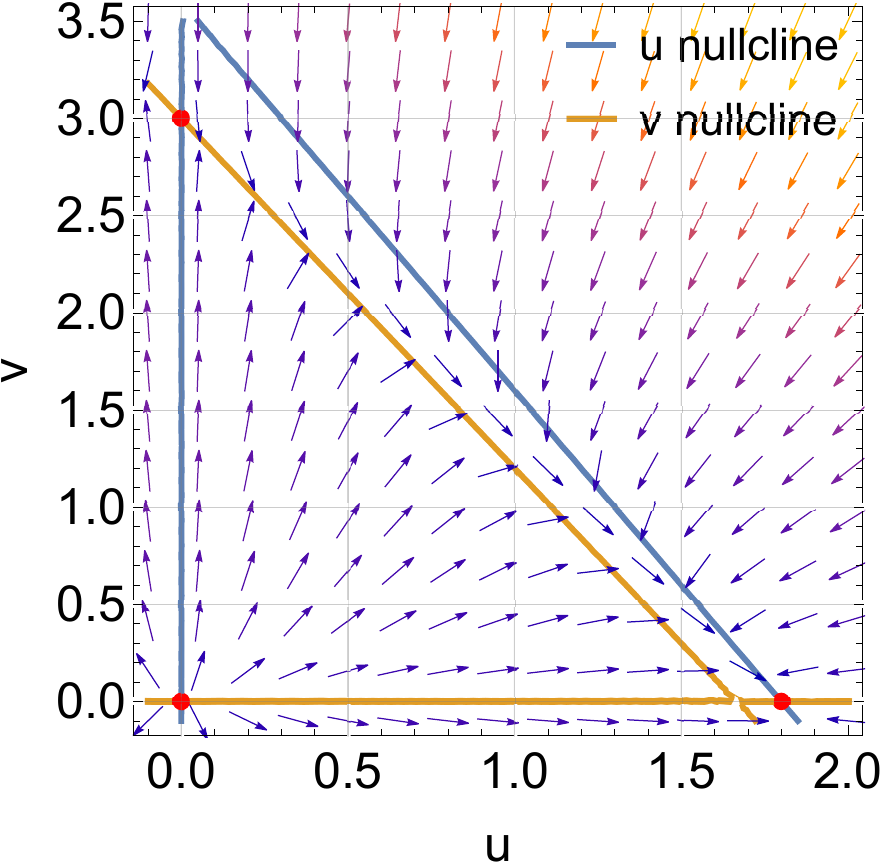}}
\subfigure[]{    
    \includegraphics[scale=.44]{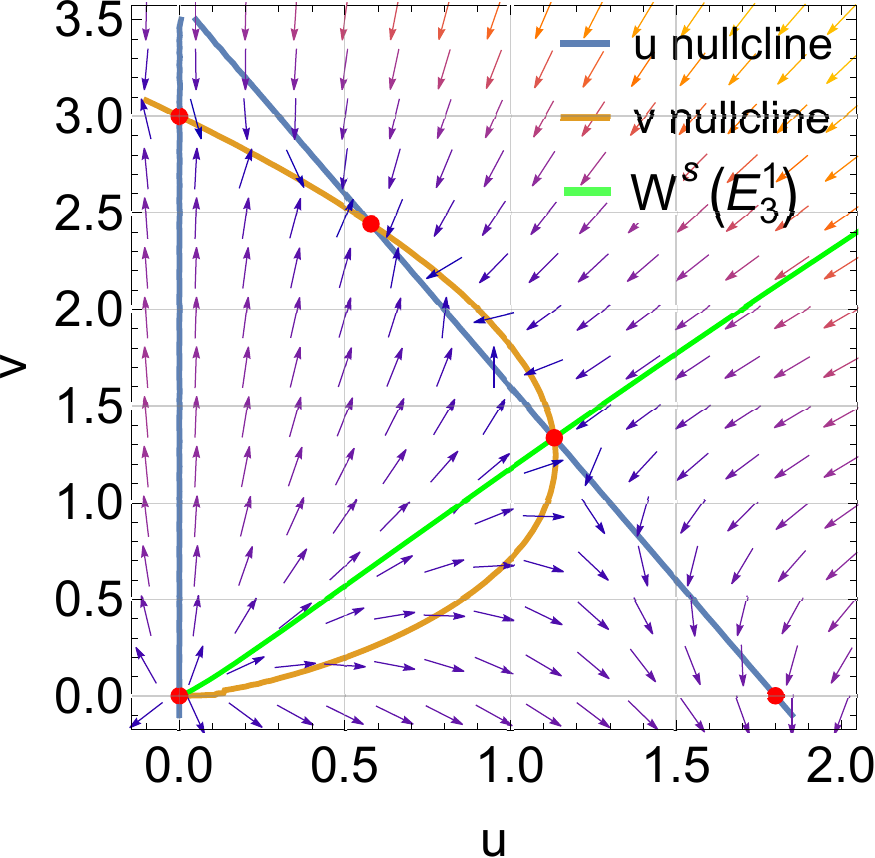}}
\subfigure[]{    
    \includegraphics[scale=.44]{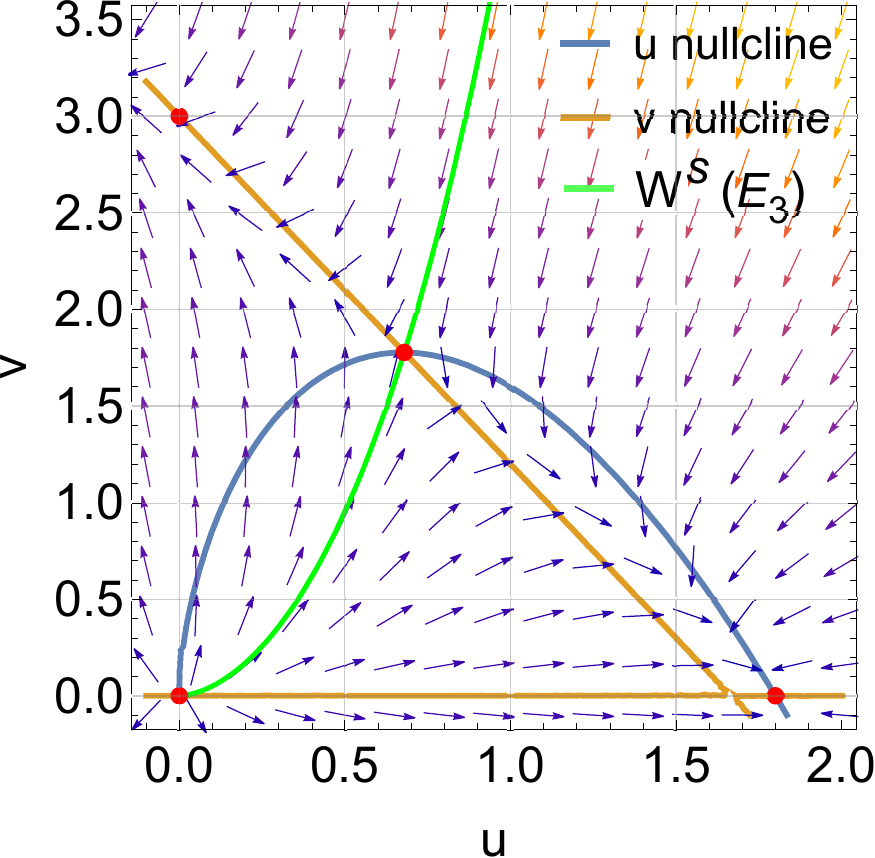}}    
\end{center}
 \caption{ Extinction case \eqref{extinctioncondition2} of model \eqref{eq:Ge1}  for $a_{1}=1.8,~a_{2}=3,~b_{1}=1,~b_{2}=1,~c_{1}=0.5,~c_{2}=1.8$. (a) classical when $p=1,~q=1$. (b) FTE when   $p=1,~q=0.3$. Here $E_3^{1}$ refers to one of the interior equilibria. (c) FTE when $p=0.4,~q=1$.}
      \label{fig:Extinction-FTE1}
\end{figure}

\begin{remark}
Note, when $0<p,q<1$ the kinetic terms are non-smooth, causing issues for uniqueness. Linearization at interior equilibrium is not effected. However, standard linearization methods do not work for boundary equilibria due to the non-smootheness. WLOG if $p=1, 0 < q < 1$, $(u^{*}, 0)$ would be attained by $v \rightarrow 0$ in a finite time, followed by an asymptotic rate of attraction to $u^{*}$. So initial data taken on the $u$-axis, can lead to non-uniqueness backwards in time. This can be circumvented if we avoid data on the $u$-axis. Such and related issues have been dealt with in \cite{P19, BS19, FCGT18, KRM20}. 
\end{remark}

We derive a sufficient condition on the initial data that yields FTE of the stronger competitor. This is 
 stated and proved via the following theorem,

\begin{theorem}
\label{thm:FiniteTimeTheorem}
Consider the competition model given by \eqref{eq:Ge1}, where $0<p<1,~ q=1$, and \eqref{extinctioncondition2} holds.
The stronger competitor $u(t)$  with initial conditions $u(0)>0,~v(0)>0$ will go extinct in finite time, if $v(0) > f(u(0))$, and trajectories will approach $(0,a_{2}/b_{2})$. Here $f$ is as in \eqref{eq:dc1}. 
However, if $(u(0), v(0))$ lies below the stable manifold $W^s(E_3)$, of the interior saddle equilibrium, then all trajectories initiating from them will approach $(a_{1}/b_{1},0)$ asymptotically.
\end{theorem}

\begin{proof}
Consider the  equation for initial condition $0 < u(0) \leq \frac{a_{1}}{b_{1}}$. Then 
\begin{equation}
 \dfrac{dv }{dt} \geq  - b_{2} v^{2} -  c_{2} uv^{q}\geq   - b_{2} v^{2} -  \dfrac{a_{1} c_{2}}{b_{1}} v.
\end{equation}
This follows as $u \leq \frac{a_{1}}{b_{1}}$, if initially so, by comparison to logistic equation. Also we consider $q = 1$.
Thus,

\begin{equation}
v(t) \geq \dfrac{a_{1} c_{2}}{e^{a_{1}c_{2}(t/b_{1}-C)}-b_{1} b_{2}}>\dfrac{a_{1} c_{2}}{e^{a_{1}c_{2}(t/b_{1}-C)}} = \left(a_{1} c_{2}e^{a_{1} c_{2} C}\right) e^{-\frac{a_{1} c_{2}}{b_{1}} t},
\end{equation}
 for all time $t \geq 0$.

 Here
$C$ is given by $\boxed{v(0) > a_{1} c_{2}e^{a_{1} c_{2} C}} $.

\begin{eqnarray}
 \dfrac{du }{dt} &=& a_{1} u - b_{1} u^{2} - c_{1}u^{p}v, \nonumber \\
&\leq& a_{1} u     - c_{1}u^{p}v.  \nonumber \\
\end{eqnarray}

Now we can divide the above by $u^{p}$ since $u$ is positive, to obtain,

\begin{equation}
 \dfrac{du }{dt}\frac {1}{u^{p}}  \leq a_{1} u ^{1-p} -  c_{1}v
\end{equation}

using the lower bound on $v$ yields,

\begin{equation}
 \dfrac{d }{dt}(u^{1-p}) \leq (1-p) a_{1} u^{1-p} - (1-p) c_{1}\left(a_{1} c_{2}e^{a_{1} c_{2} C}\right) e^{-\frac{a_{1} c_{2}}{b_{1}} t},
\end{equation}

multiplying both sides by the integrating factor $e^{-(1-p) a_{1} t}$,
and subsequently integrating the above in the time interval $[0,t]$ with $t \leq T^{*}$, we obtain

\begin{eqnarray}
&&e^{-(1-p) a_{1} t}u^{1-p} \nonumber \\
&\leq&  (u(0))^{1-p} - c_{1}\left( \dfrac{ (1-p)  b_{1} v(0)}{a_{1} c_{2}+(1-p)a_{1} b_{1}}\right) (1-e^{-(\frac{a_{1} c_{2}}{b_{1}}+(1-p)a_{1})t}). \nonumber \\
\end{eqnarray}

%

Which then implies the finite time extinction of $u$, if 

\begin{equation}
 (u(0))^{1-p} < c_{1}\left( \dfrac{ (1-p)  b_{1} v(0)}{a_{1} c_{2}+(1-p)a_{1} b_{1}}\right).
\end{equation}

Thus we choose $f$ according to

\begin{equation}
\label{eq:dc1}
f (u(0)) =\left(\dfrac{a_{1} c_{2}+(1-p)a_{1} b_{1}}{(1-p) c_{1} b_{1}}\right) (u(0))^{1-p}
\end{equation}


and for initial data chosen s.t $v(0) \geq f(u(0))$, $u$ will go extinct in finite time.
This proves the theorem.

\end{proof}

 We provide some simulations next to elucidate. 
 
 \begin{figure}[!htb]
\begin{center}
 \subfigure[]{    
    \includegraphics[scale=.45]{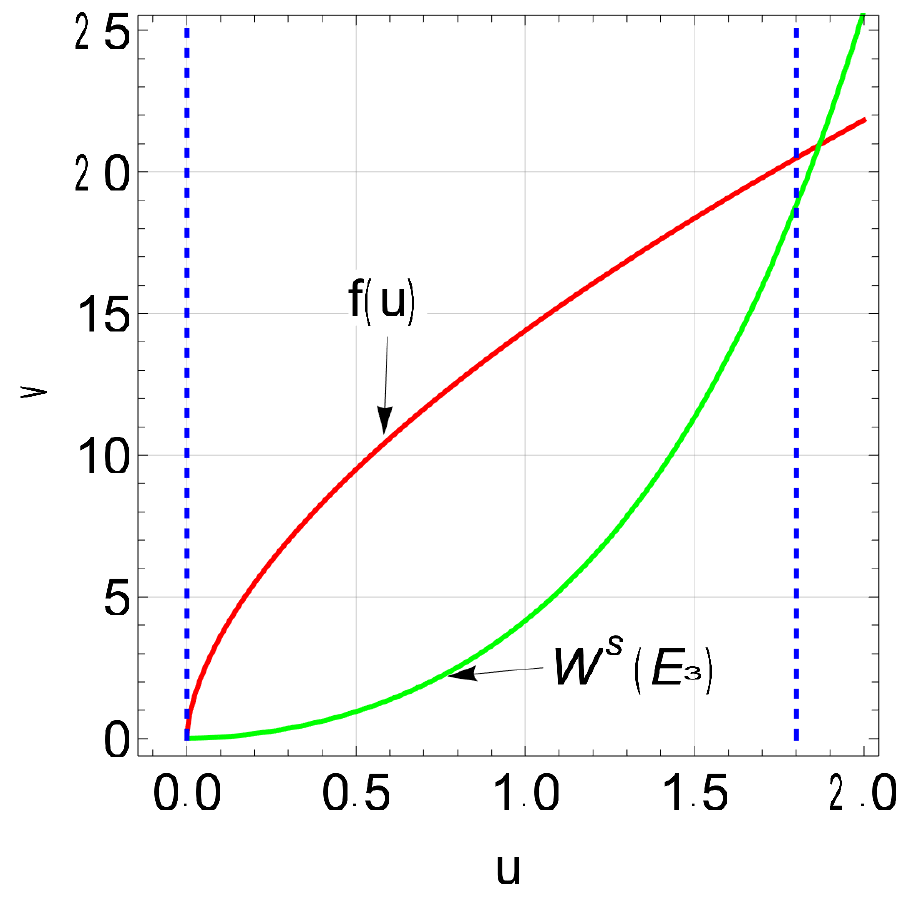}}
 \subfigure[]{    
    \includegraphics[scale=.45]{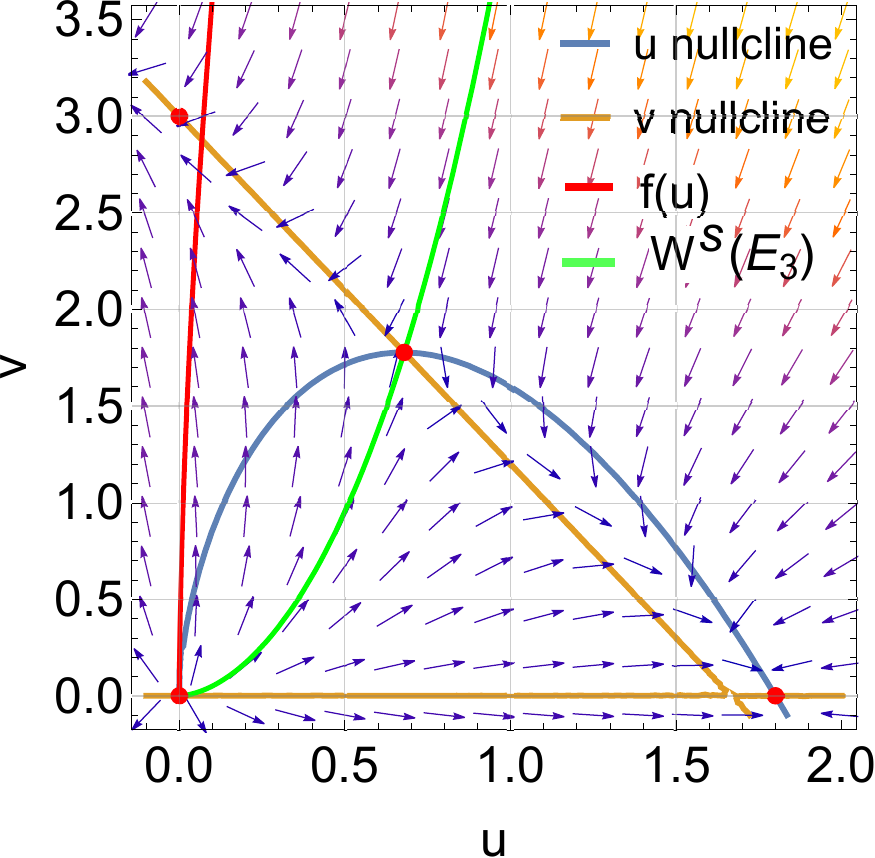}}
\end{center}
 \caption{Herein we demonstrate our results for Theorem \ref{thm:FiniteTimeTheorem}. Here   $a_{1}=1.8,~a_{2}=3,~b_{1}=1,~b_{2}=1,~c_{1}=0.5,~c_{2}=1.8,~p=0.4,~q=1$.}
      \label{fig:Separatrix1}
\end{figure}

%

\subsection{The Weak Competition/Co-existence Case}
Here we consider the case 

\begin{align}\label{weakcondition}
\dfrac{b_{1}}{c_{2}}>\dfrac{a_{1}}{a_{2}}>\dfrac{c_{1}}{b_{2}}.
\end{align}

 The classical theory for $p=1$, predicts that all initial conditions would be attracted to a interior equilibrium. In this setting the competitors $u$ and $v$ coexist. However, this is not the case if $0<p<1$.

We state the following theorem.
\begin{theorem}
\label{thm:FiniteTimeTheoremWeakCase}
Consider the competition model given by \eqref{eq:Ge1}, where $0<p<1,~ q=1$, and \eqref{weakcondition} holds.
The competitor $u(t)$  with initial conditions $u(0)>0,~v(0)>0$ will go extinct in finite time, if $v(0) > f(u(0))$, and trajectories will approach $(0, a_{2}/b_{2})$. Here $f$ is same as in Theorem \ref{thm:FiniteTimeTheorem}. 
However, if $(u(0), v(0))$ lies below the stable manifold $W^s(E_3^1)$ of the interior equilibrium, then all trajectories will approach the stable interior equilibrium $E_3^2$.
\end{theorem}
The proof is as of Theorem \ref{thm:FiniteTimeTheorem}.
 

\begin{conjecture}\label{conjecture:C1}
Assume the classical competition model when $p=q=1$ in model \eqref{eq:Ge1}, and the condition in \eqref{weakcondition} holds true, here the competitors coexist. There is a critical window of parameter $c_1\in [c_1^{*},c_1^{**}]$, for which there are two interior equilibria for $0<p<1$ and no interior equilibrium for $0<q<1$. Furthermore, there is a critical window of parameter $c_1\in [c_1^{***},c_1^{****}]$, for which there is no interior equilibrium for $0<p<1$ and there are two interior equilibria for $0<q<1$.
\end{conjecture}

We provide some simulations next to elucidate the conjecture. 

\begin{figure}[!htb]
\begin{center}
\subfigure[]{
    \includegraphics[scale=.45]{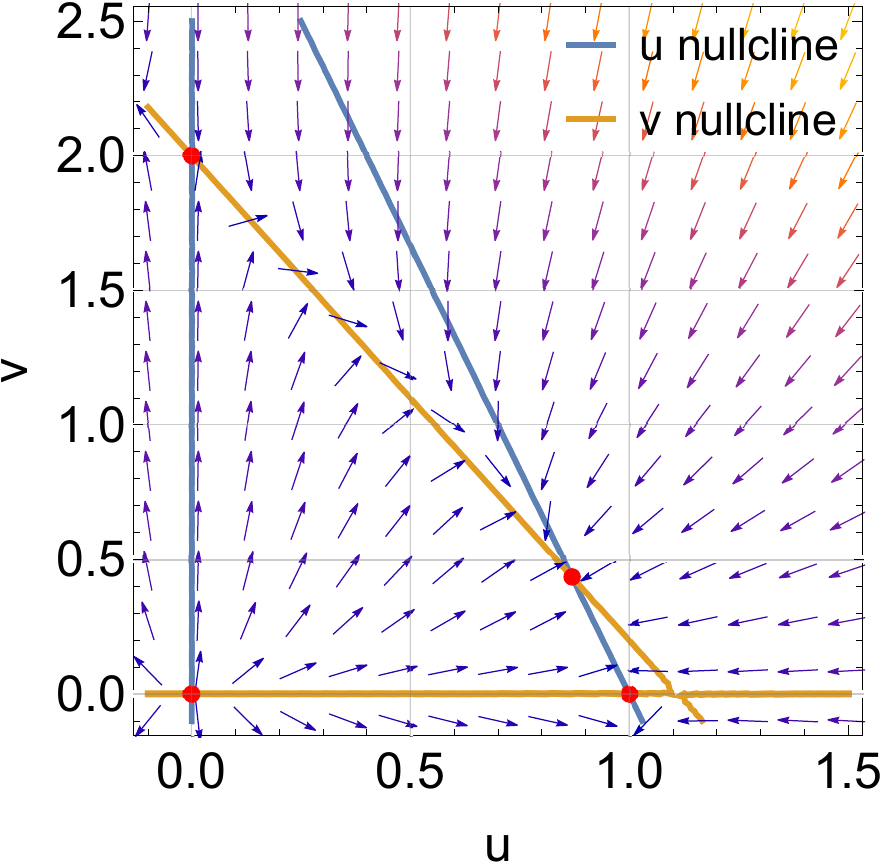}}
\subfigure[]{    
    \includegraphics[scale=.45]{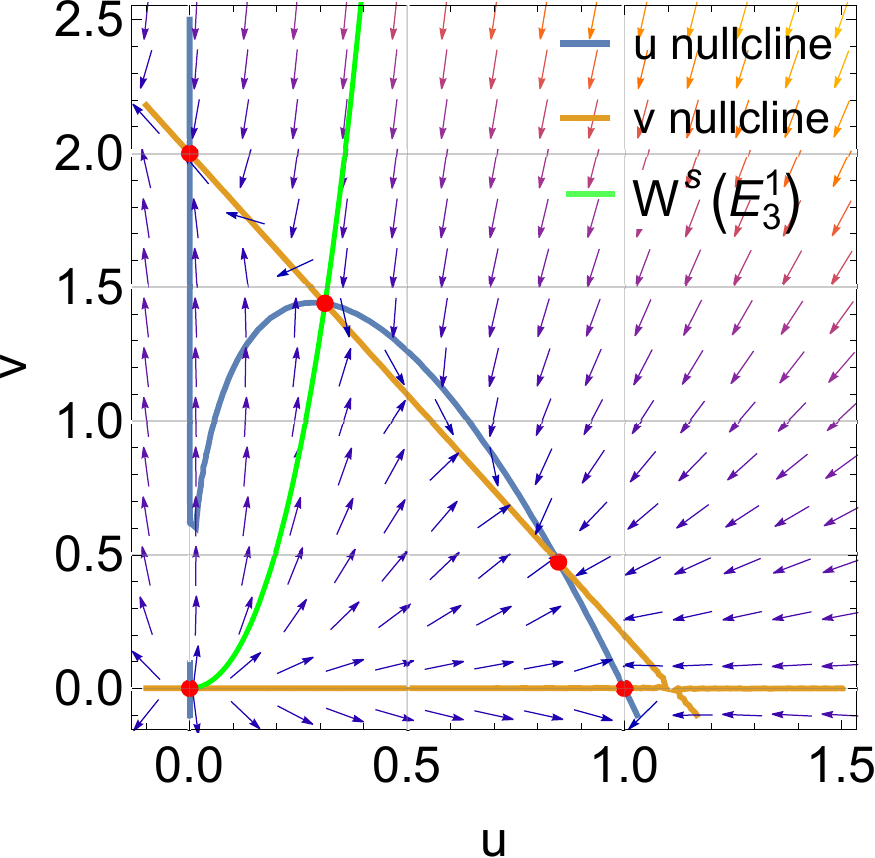}}
\subfigure[]{    
    \includegraphics[scale=.45]{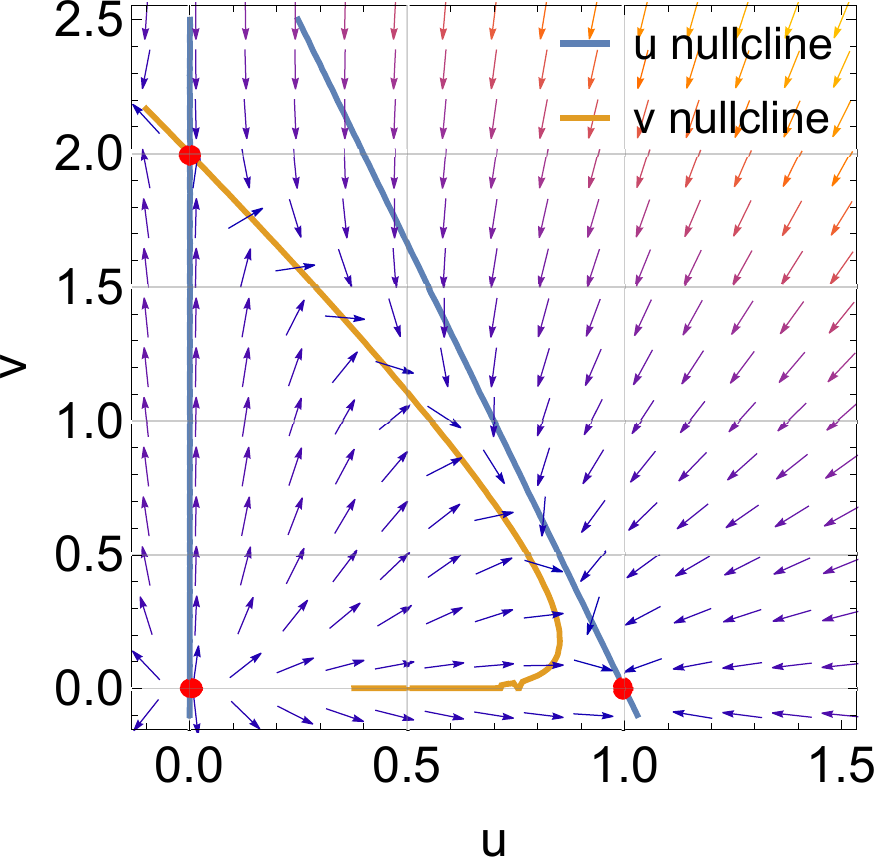}}    
\end{center}
 \caption{Weak competition case \eqref{weakcondition} of model \eqref{eq:Ge1} for $a_{1}=1,~a_{2}=2,~b_{1}=1,~b_{2}=1,~c_{1}=0.3,~c_{2}=1.8$. (a) classical case when $p=1,~q=1.$ (b)  FTE when  $p=0.6,~q=1.$ (c)  FTE when  $p=1,~q=0.9$.}
      \label{fig:Weak-FTE1}
\end{figure}

\begin{figure}[!htb]
\begin{center}
\subfigure[]{
    \includegraphics[scale=.45]{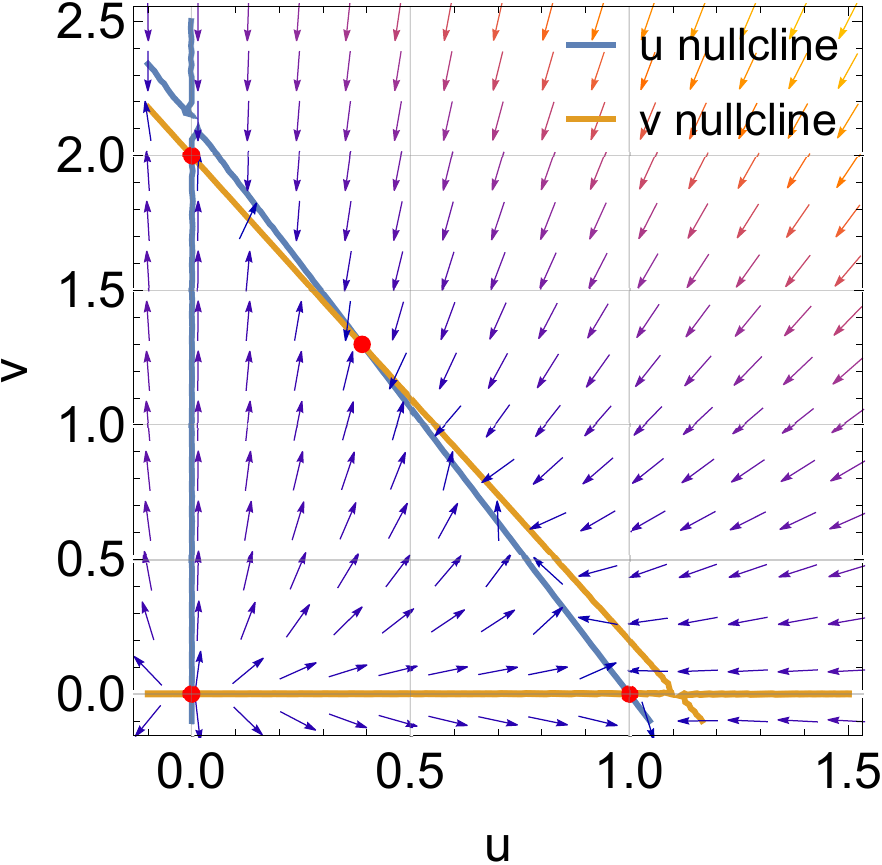}}
\subfigure[]{    
    \includegraphics[scale=.45]{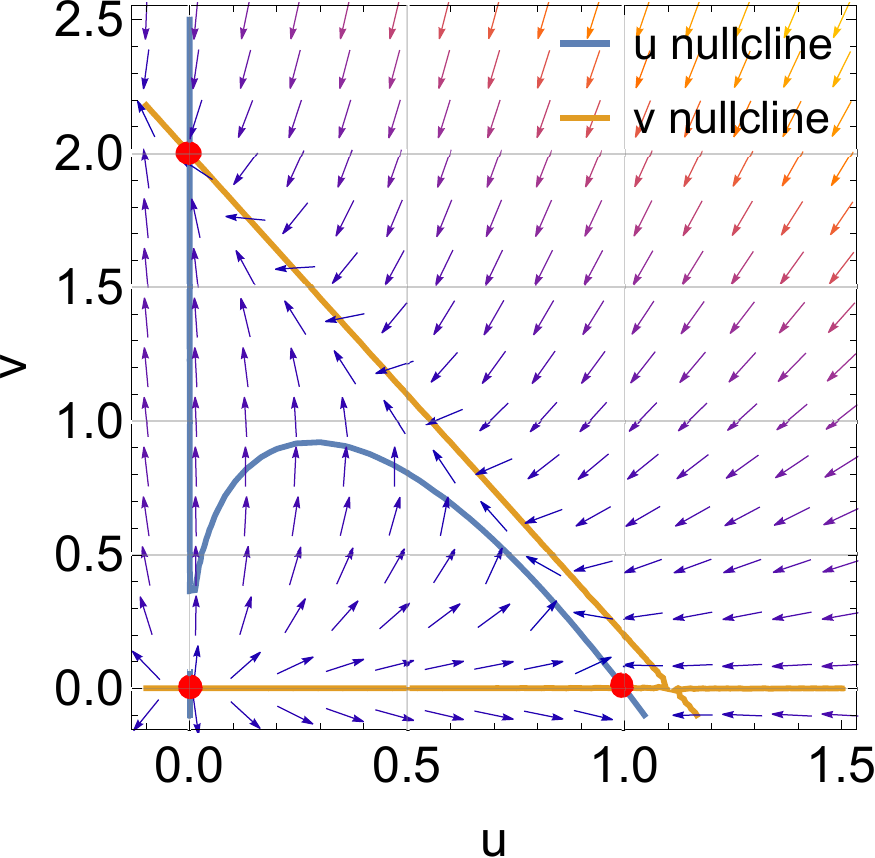}}
\subfigure[]{    
    \includegraphics[scale=.45]{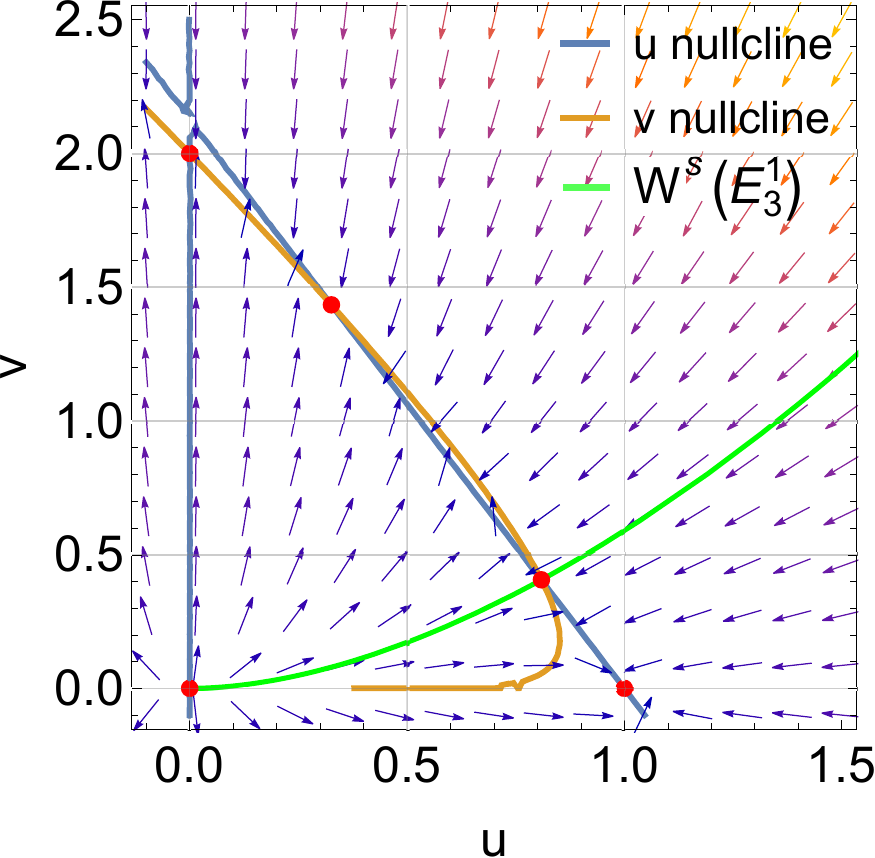}}    
\end{center}
 \caption{  Weak competition case \eqref{weakcondition} of model \eqref{eq:Ge1} for $a_{1}=1,~a_{2}=2,~b_{1}=1,~b_{2}=1,~c_{1}=0.47,~c_{2}=1.8$. (a) classical case when $p=1,~q=1.$ (b)  FTE when  $p=0.6,~q=1.$ (c)  FTE when  $p=1,~q=0.9$.}
      \label{fig:Weak-FTE2}
\end{figure}


\begin{figure}[!htb]
\begin{center}
\subfigure[]{
    \includegraphics[scale=.44]{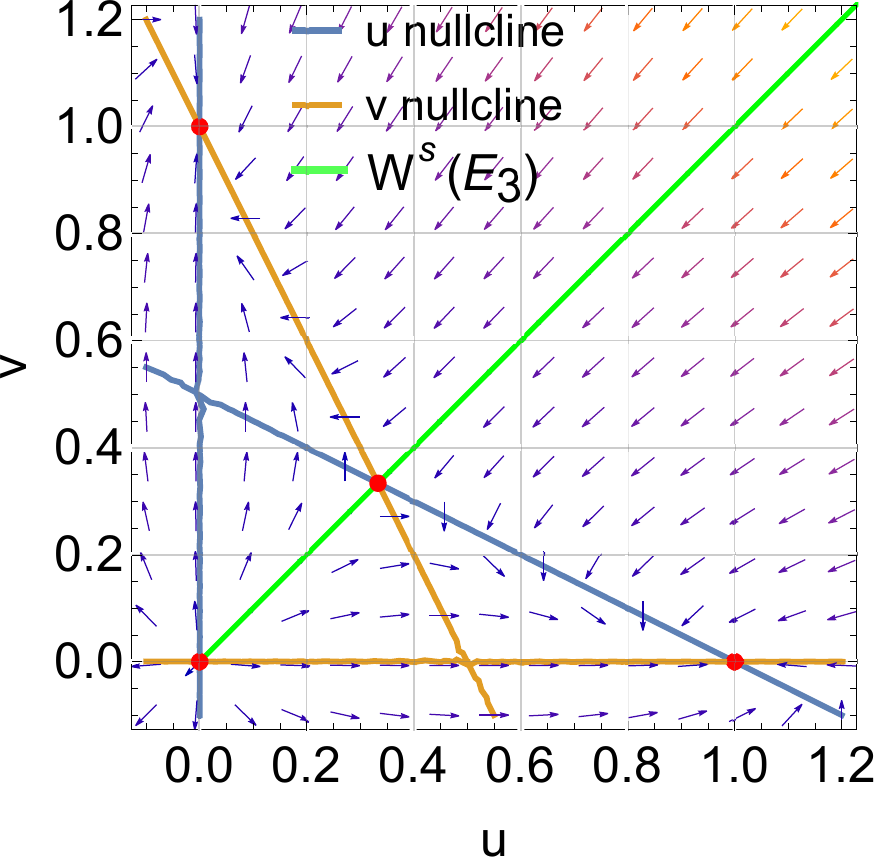}}
\subfigure[]{    
    \includegraphics[scale=.44]{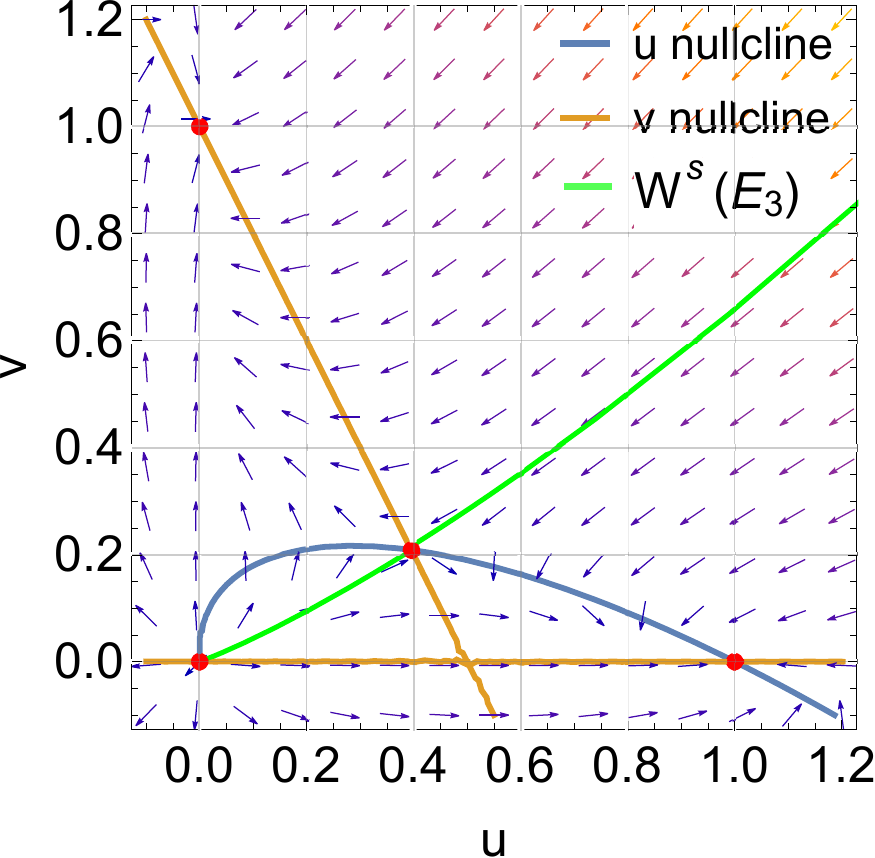}}
\subfigure[]{    
    \includegraphics[scale=.44]{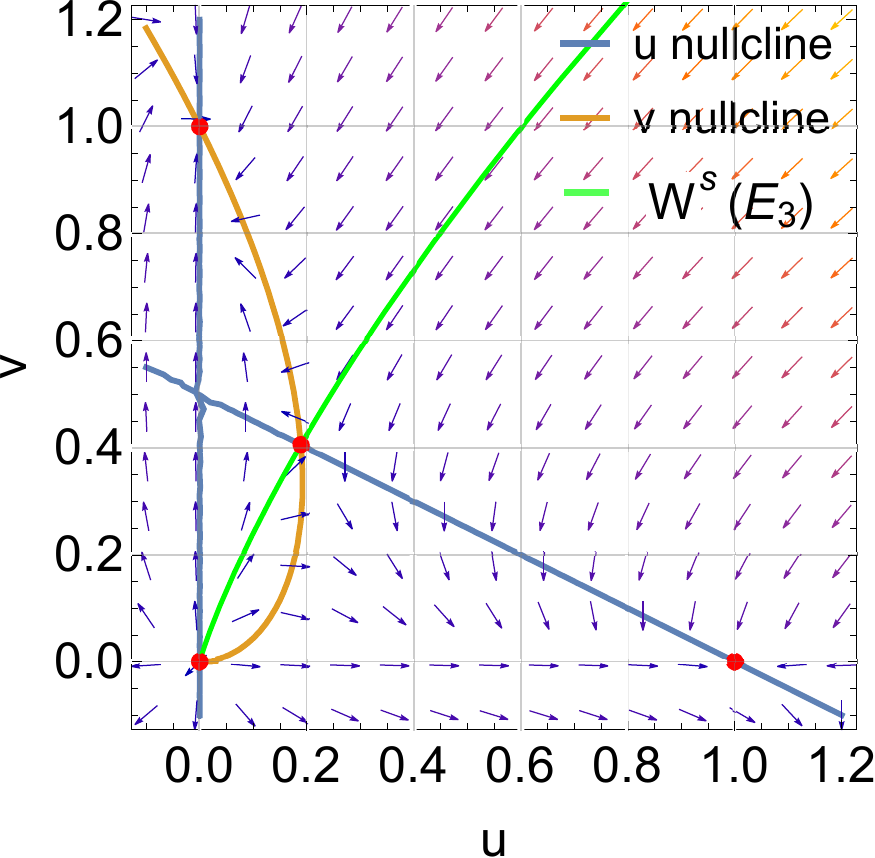}}
\end{center}
  \caption{ Strong competition case \eqref{eq:f1} of model \eqref{eq:Ge1} for $a_{1}=1,~a_{2}=1,~b_{1}=1,~b_{2}=1,~c_{1}=2,~c_{2}=2$ (a) classical when $p=1,~q=1$ (b) FTE when $p=0.6,~q=1$ (b) FTE when $p=1,~q=0.5$.}
      \label{fig:Strong-FTE1}
\end{figure}

\section{The PDE case}

\subsection{The case of strong competition}
The spatially explicit two species competition model has been intensely investigated \cite{Ninomiya1995, Cantrell2003,  Chen2020, DeAngelis2016, Lou2017, Mazari2020a, Liang2012, Yao2008, Hutson2001, Ni2011, Ni2020, Ni2012}.
We consider a generalized version
\begin{equation}
\label{eq:Ge1p}
\left\{ \begin{array}{ll}
\dfrac{\partial u }{\partial t} &~ =d_1 \Delta u +  a_{1}u- b_{1}u^{2} - c_{1}u^{p}v , 0 <p \leq 1, \\[2ex]
\dfrac{\partial  v}{ \partial t} &~ = d_2 \Delta v + a_{2} v- b_{2} v^{2} - c_{2} u v,
\end{array}\right.
\end{equation}

\begin{equation}
\label{eq:Ge1b}
\nabla  u \cdot n = \nabla  v \cdot n  = 0, on \ \partial \Omega\ , \ u(x,0) =  u_{0}(x) > 0, \  v (x,0) =  v_{0}(x) > 0,
\end{equation}
here we consider a bounded domain $\Omega \subset \mathbb{R}^{n}, \ n=1, 2$.  
Under the strong competition setting, 
\begin{equation}
\label{eq:f1}
\frac{ b_{1}}{c_{2}}  < \frac{a_{1} }{a_{2}} <  \frac{ c_{1}}{b_{2}}. 
\end{equation}
When $p=1$, classical results show that in the absence of diffusion, there is a stable manifold of the saddle equilibrium (separatrix) denoted as $h$, that splits the phase space into 2 regions, $W_{B}$ the region above the separatrix - Note, for initial data $(u(x,0), v(x,0)) \in W_{B}$ $ \forall x \in \Omega$, the solution converges to $(0, \frac{a_{2}}{b_{2}})$. Likewise, $W_{A}$ is the region below the separatrix, and for initial data $(u(x,0), v(x,0)) \in W_{A}$ $ \forall x \in \Omega$, the solution converges to $( \frac{a_{1}}{b_{1}}, 0)$.
We recap a classical result from \cite{Masato1998, Ninomiya1995}, to this end.

\begin{theorem}[Diffusion induced extinction]
\label{thm:die}
 Let $(u, v)$ be a solution of \eqref{eq:Ge1p}-\eqref{eq:Ge1b}, and $p=q=1$. Suppose that $h^{''} \leq 0$. Then there exists initial data such that $(u(x,0), v(x,0)) \in W_{B}$ $ \forall x \in \Omega$, but the solution initiating from this data converges uniformly to $(\frac{a_{1}}{b_{1}},0)$.
\end{theorem}
This can change when the FTE dynamic is present. 

\begin{remark}
Note, the FTE dynamic could hinder well posedness due to the non-smooth term $u^{p}$, $0<p<1$, in \eqref{eq:Ge1p}. Two species semi-linear reaction diffusion systems have been considered in \cite{BS02}, where there are non-smooth terms in \emph{one} of the equations - such as in our case. 
The key tool used to show existence of bounded global in time, classical solutions, is a weak comparison principle method \cite{BS02}. This is for the dirichlet boundary condition however. Recently such problems have also been investigated in the case of more complicated boundary conditions, \cite{KF18}. In general, there could be data that lead to non-unique solutions,  however, for certain given data, one has weak/classical solutions to the class of problems considered herein \cite{BV03}, even for the neuman problem. Our goal is not to demonstrate well (or ill) posedness here, more to focus on the dynamical changes that the $0<p<1$ can bring about, and the many ecological consequences therein.
\end{remark}

We state and prove the following result,

\begin{theorem}[Finite time extinction induced recovery]
 \label{thm:FTE-ESTIMATE}
 Consider \eqref{eq:Ge1p}-\eqref{eq:Ge1b}, under the strong competition case \eqref{eq:f1}, when $p=1$, and initial data $(u_{0}(x), v_{0}(x)) \in W_{B}$ $ \forall x \in \Omega$, that converges uniformly to $( \frac{a_{1}}{b_{1}}, 0)$, for some $d_{2}, d_{1 } > 0$. Then $\exists$ $p <1$, 
 s.t solutions to \eqref{eq:Ge1p}-\eqref{eq:Ge1b}, from the same 
initial data $(u_{0}(x), v_{0}(x)) \in W_{B}$ converge uniformly to $(0, \frac{a_{2}}{b_{2}})$, as long as equations \eqref{eq:cond1}-\eqref{eq:cond123} hold.
\end{theorem}

\begin{figure}[hbt!] 
\begin{center}
\subfigure[]{
    \includegraphics[scale=.31]{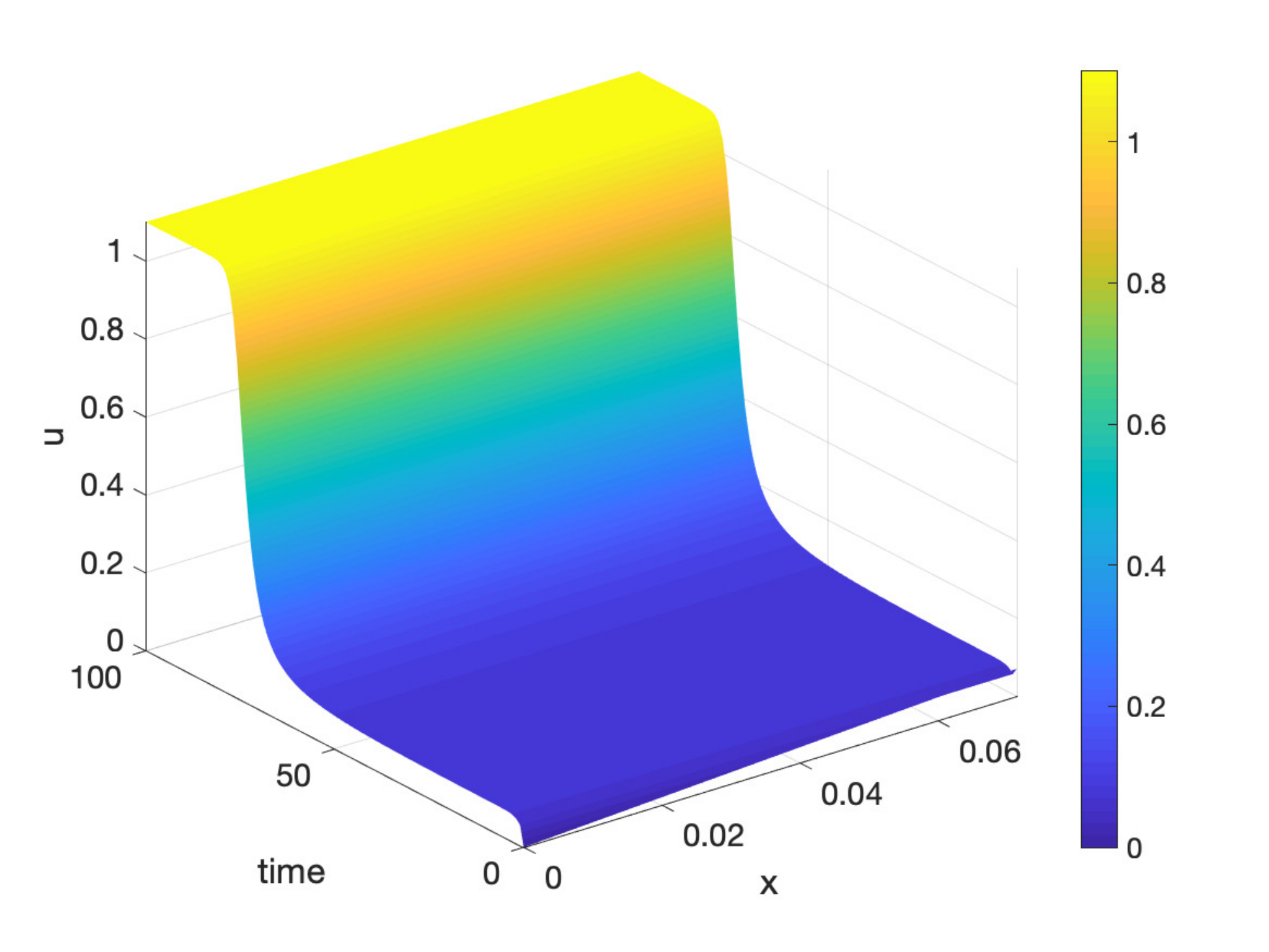}}
\subfigure[]{    
    \includegraphics[scale=.31]{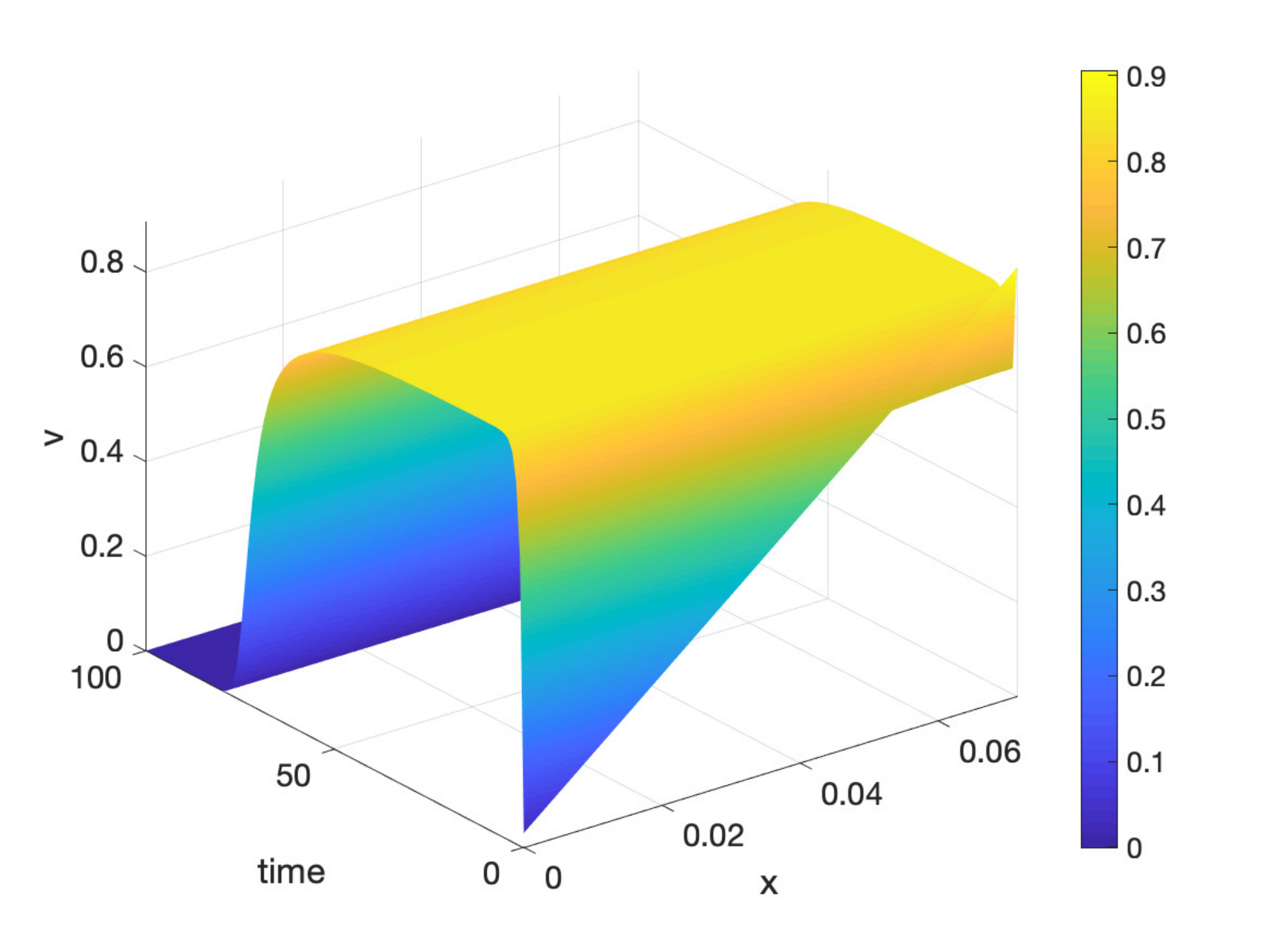}}
\end{center}
\caption{Diffusion induced extinction is seen. Here we consider parameters  $a_1=1.1, b_1=1, c_1=1.2, a_2=1, b_2=1, c_2=2, p=1, d_1=1, d_2=0.001$ for  \eqref{eq:Ge1p}-\eqref{eq:Ge1b}.  We choose $\Omega= [0, 0.071429]$. The initial data is chosen as per the estimates of Theorem \ref{thm:FTE-ESTIMATE}, see Fig. \ref{fig:estimates}.}
\label{fig:Est2}
\end{figure}

\begin{figure}[hbt!] 
\begin{center}
\subfigure[]{
    \includegraphics[scale=.31]{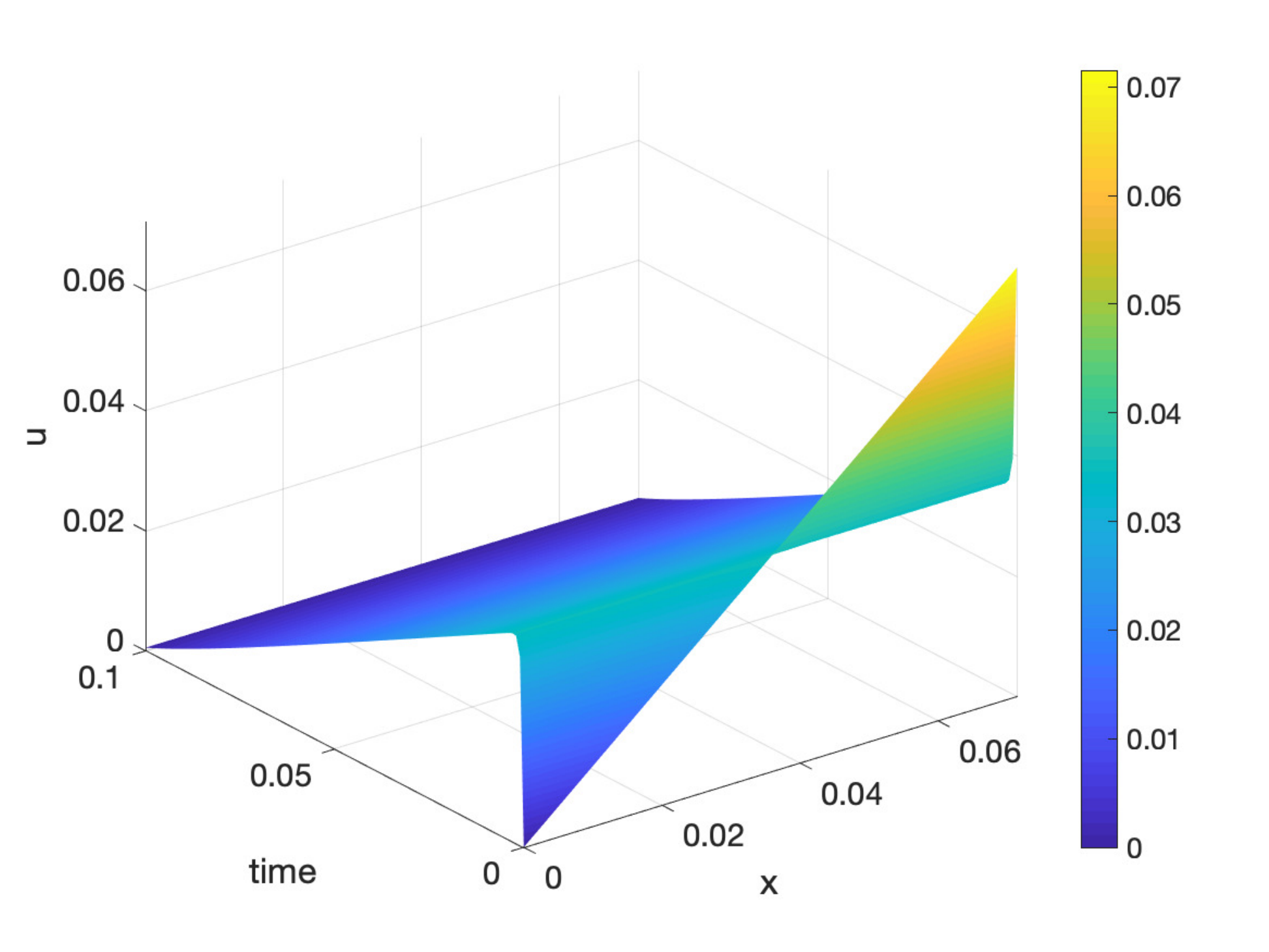}}
\subfigure[]{    
    \includegraphics[scale=.31]{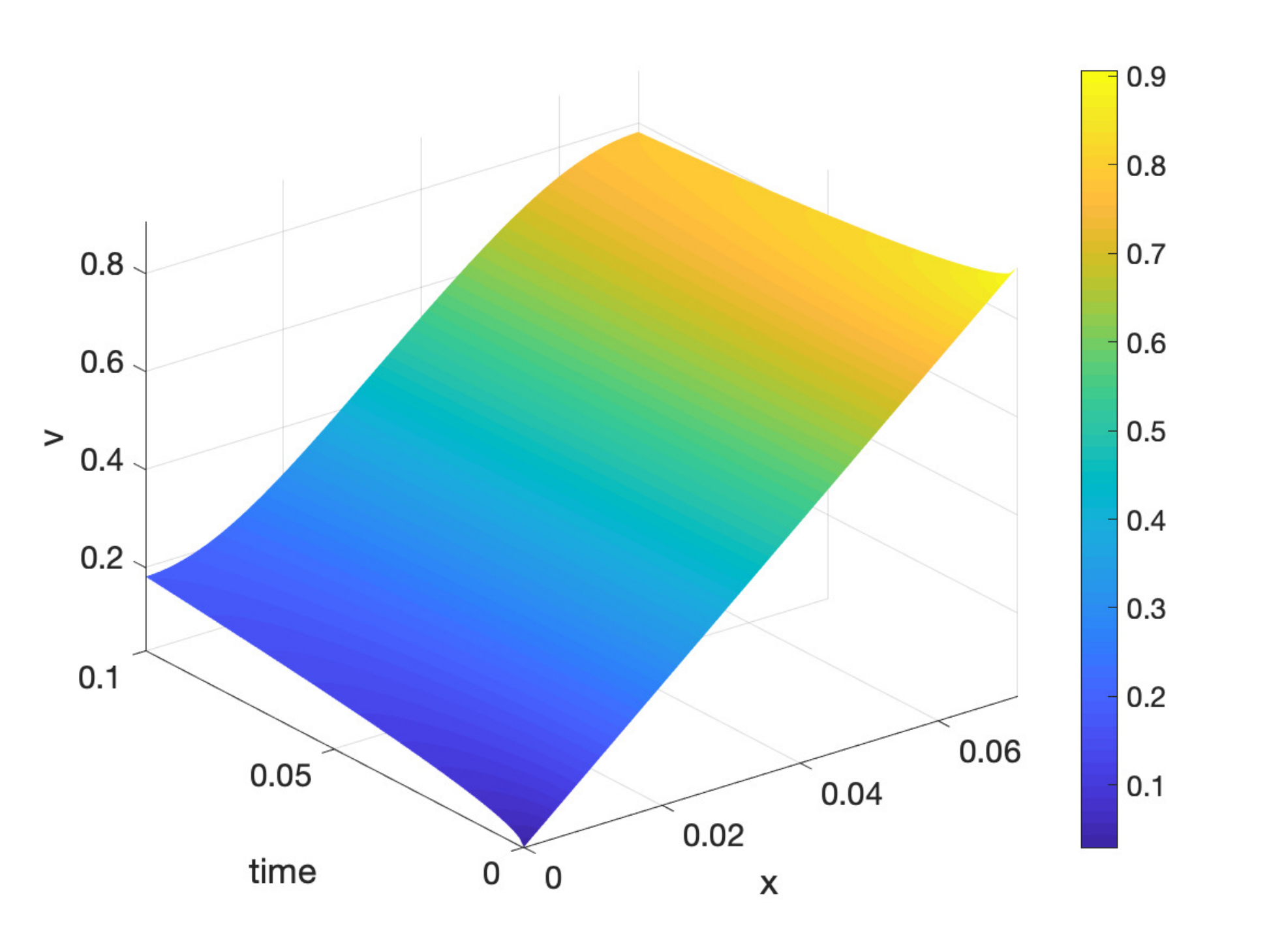}}
\end{center}
\caption{Finite time induced recovery of superior species. We choose  $a_1=1.1, b_1=1, c_1=1.2, a_2=1, b_2=1, c_2=2, p=0.1, d_1=1, d_2=0.001$ and $\Omega= [0, 0.071429]$ for  \eqref{eq:Ge1p}-\eqref{eq:Ge1b}. The initial data is chosen as per the estimates of Theorem \ref{thm:FTE-ESTIMATE}, see Fig. \ref{fig:estimates}.}
\label{fig:Est3}
\end{figure}

\begin{proof}

%
%

For the constant coefficient case as we are dealing with herein, solutions are spatially homogeneous \cite{Masato1998, Ninomiya1995}, thus our system is reduced to

\begin{equation}
\label{eq:Ge1pno}
\left\{ \begin{array}{ll}
\dfrac{\partial u }{\partial t} &~ =    a_{1}u - b_{1}u^{2} - c_{1}u^{p}v ,\\[2ex]
\dfrac{\partial  v}{ \partial t} &~ =   a_{2} v -  b_{2}v^{2} - c_{2} uv .
\end{array}\right.
\end{equation}

%
%
%
%
%
%

Standard estimates as in theorem \ref{thm:FiniteTimeTheorem}, yield the finite time extinction of $u$ for initial data chosen s.t.

\begin{equation}
\left(\dfrac{a_{1} c_{2}+(1-p)a_{1} b_{1}}{(1-p) c_{1} b_{1}}\right) (u_{0}(x))^{1-p} = f_{1}(u_{0}) \leq v_{0}(x).
\end{equation}

Note, analysis of the ODE/kinetic system, via a simple modification of Lemma \ref{lem:l1}, see Fig. \ref{fig:Extinction-FTE1}, clearly shows that when $p<1$, the interior equilibrium is lowered, and so is the separatrix. We refer to the separatrix for the $0<p<1$ case as $h_{1}$.
Now consider when $p=1$, initial data $(u_{0}(x), v_{0}(x)) \in W_{B}$, for which diffusion induced extinction occurs.
Since $(u(x,0), v(x,0)) \in W_{B}$, it lies above the separatrix $h$, and by the concavity assumption on $h$, $h$ lies above the line segment connecting $(0,0)$ and $(u^{*}, v^{*})$, which is given by the equation

\begin{equation}
v_{0}(x) = f_{2}(u_{0}) = \left(\frac{c_{2}a_{1} - b_{1}a_{2}}{c_{1}a_{2} - a_{1}b_{2}}\right) u_{0}(x).
\end{equation}

Thus if we choose $p$ s.t $h_{1}$ is lowered enough s.t $f_{1}(u_{0})< f_{2}(u_{0})$, for certain $u^{*}_{0}$, then there exists data $(u^{*}_{0}, v^{*}_{0}) \in W_{B}$ (for which diffusion induced extinction occurs if $p=1$), but that lies above the separatrix $h$, which in turn lies above $ f_{2}(u_{0}) $, which by the appropriate choice of $p < 1$ lies above $f_{1}(u_{0})$, which lies above $h_{1}$ - and so will converge uniformly to $(0, v^{*})$, and diffusion induced extinction does not occur, when $p < 1$. To this end it is sufficient that,

\begin{equation}
\label{eq:cond1}
f_{1}(u_{0}(x)) = \left(\dfrac{a_{1} c_{2}+(1-p)a_{1} b_{1}}{(1-p) c_{1} b_{1}}\right) (u_{0}(x))^{1-p} \leq v_{0}(x) \leq \left(\frac{c_{2}a_{1} - b_{1}a_{2}}{c_{1}a_{2} - a_{1}b_{2}}\right) u_{0}(x) = f_{2}(u_{0}(x))
\end{equation}
and,
\begin{equation}
\label{eq:cond12}
u_{0}(x) \leq u^{*} = \left(\frac{c_{1}a_{2} - a_{1}b_{2}}{c_{2}c_{1} - b_{1}b_{2}} \right).
\end{equation}

A sufficient parametric restriction for which the above is true is given by

\begin{equation}
\label{eq:cond123}
\left(\dfrac{a_{1} c_{2}+(1-p)a_{1} b_{1}}{(1-p) c_{1} b_{1}}\right)  \leq \left(\frac{c_{2}a_{1} - b_{1}a_{2}}{c_{1}a_{2} - a_{1}b_{2}}\right)  \left(\frac{c_{1}a_{2} - a_{1}b_{2}}{c_{2}c_{1} - b_{1}b_{2}} \right)^{p}.
\end{equation}

This proves the theorem.

\end{proof}

\begin{figure}[hbt!]
\begin{center}
    \includegraphics[scale=.17]{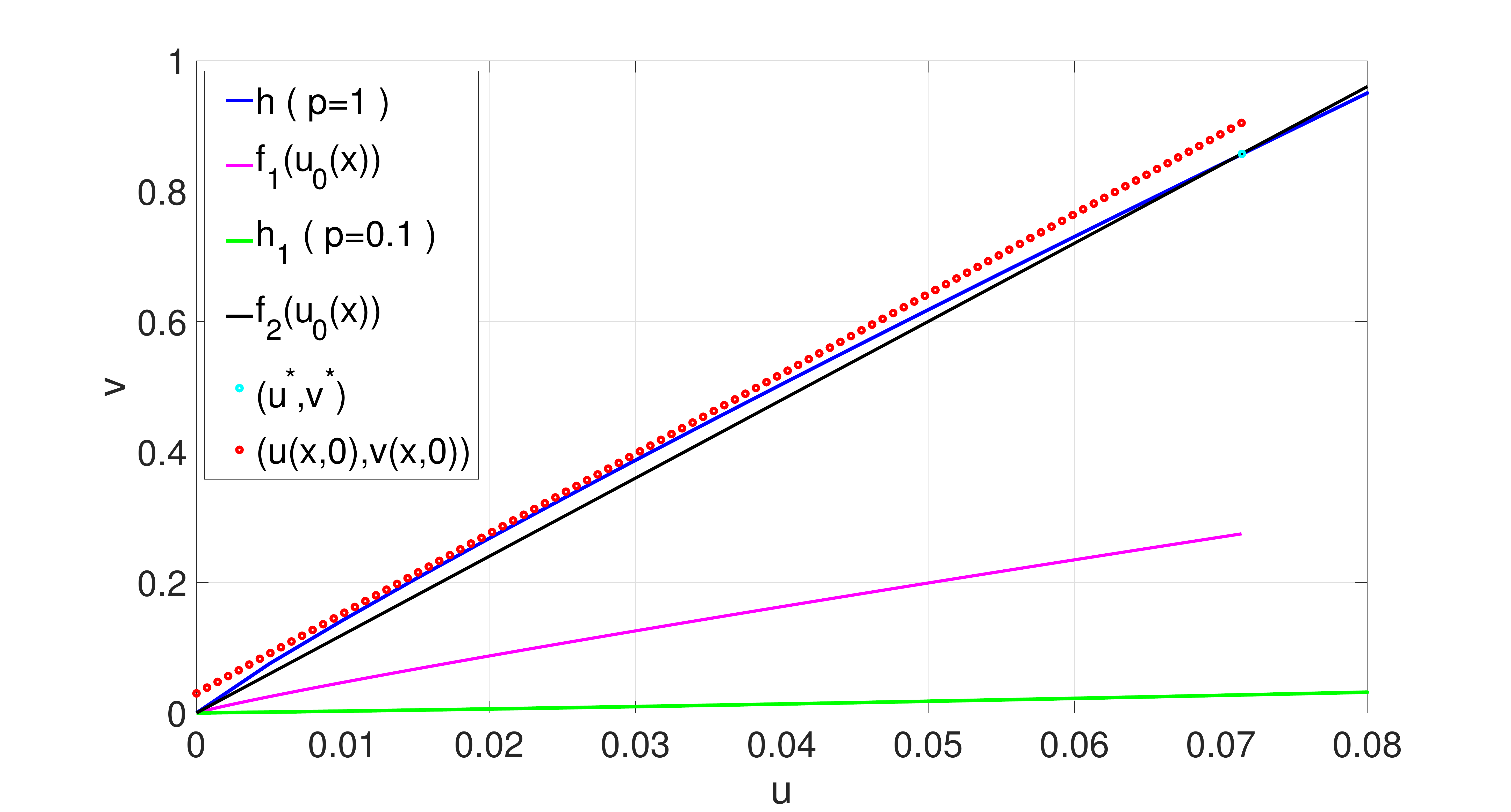}
\end{center}
\caption{Plot validating theorem \ref{thm:FTE-ESTIMATE}. Parameters used are $a_1=1.1, b_1=1, c_1=1.2, a_2=1, b_2=1, c_2=2, p=0.1$. We choose $\Omega = [0, 0.071429]$. Here, $(u^*,v^*)=(0.071429,0.85714) $. The red dots are the data, that lie above the separatrix when $p=1$, we see diffusion induced extinction occur, that is we approach $(\frac{a_{1}}{b_{1}},0)$ in this case - see Fig. \ref{fig:Est2}. When $p<1$, the same data converges to $(0, \frac{a_{2}}{b_{2}})$, see Fig. \ref{fig:Est3}.}
\label{fig:estimates}
\end{figure}

\subsection{The Spatially Inhomogeneous Problem }

The spatially inhomogeneous problem has been intensely investigated in the past 2 decades \cite{Cantrell2003,  Chen2020, DeAngelis2016, He2019, He2013a, He2013b, He2016a, He2016b, Hutson2003, Lou2006a, Lou2006b, Liang2012, Li2019, Lam2012, Nagahara2018, Nagahara2020}. The premise here is that $u,v$ do not have resources that are uniformly distributed in space, rather there is a spatially dependent resource function $m(x)$. We consider again a normalized generalization of the classical formulation, where there are 2 parameters $b,c$ for inter/intra specific kinetics, as opposed to 6 from earlier. The parameter $p$, enables FTE in $u$.

\begin{equation}
\label{eq:Ge1pn}
\left\{ \begin{array}{ll}
\dfrac{\partial u }{\partial t} &~ = d_{1}\Delta u +  m(x)u - u^{2} - bu^{p}v ,   0 < p \leq 1, \\[2ex]
\dfrac{\partial  v}{ \partial t} &~ =  d_{2}\Delta v + m(x) v -  v^{2} - cuv    ,
\end{array}\right.
\end{equation}

\begin{equation}
\label{eq:Ge1bh}
\nabla  u \cdot n = \nabla  v \cdot n  = 0, on \ \partial \Omega\ , \ u(x,0) =  u_{0}(x) > 0, \  v (x,0) =  v_{0}(x) > 0.
\end{equation}
Note, $p=1$, is the classical case. We consider $m$ to be non-negative on $\Omega$, and bounded. We recap a seminal classical result \cite{Dockery1998, Hastings},

\begin{theorem}[Slower diffuser wins]
Consider \eqref{eq:Ge1pn}-\eqref{eq:Ge1bh}, when $b=c=p=1$, and $d_{1} < d_{2}$, solutions initiating from any positive initial data $(u_{0}(x), v_{0}(x))$ converge uniformly to $(u^{*}(x),0)$. 
\end{theorem}
That is, the slower diffuser wins, in the case of equal kinetics.
However, a difference in the inter specific kinetics can cause the slower diffuser to \emph{loose}, depending on the initial conditions. We now state the following result in one spatial dimension,

\begin{theorem}[Slower diffuser can loose]
\label{thm:sdl}
Consider \eqref{eq:Ge1pn}-\eqref{eq:Ge1bh}, where $\Omega \subset \mathbb{R}$, when $b=c=p=1$, $d_{1} < d_{2}$. There exists positive initial data $(u_{0}(x), v_{0}(x))$, for which solutions converge to $(u^{*}(x),0)$, but solutions with the same diffusion coefficients, initiating from the same data, will converge to $(0, v^{*}(x))$ in finite time, for a sufficiently chosen $0< p < 1$.
\end{theorem}


\begin{proof}
Via comparison with the logistic equation \cite{Cantrell2003}, we see that $u \leq C m(x)$, $\forall x, t \in \Omega \times [0,\infty)$. Now from the equation for $v$ in \eqref{eq:Ge1pn}, we have,

\begin{equation}
\dfrac{\partial  v}{ \partial t}  =  d_{2}\Delta v + m(x) v -  v^{2} - uv \geq d_{2}\Delta v -  v^{2} - C||m||_{\infty}v,
\end{equation}

via comparison we have, 

\begin{equation}
\label{eq:v1}
v(x,t)  \geq  C_{1}v_{0}(x) e^{-C_{2}t}
\end{equation}

where $C_{1}, C_{2}$, are independent of $u$. We now multiply the $u$ equation in \eqref{eq:Ge1pn} by $u$ and integrate by parts to obtain

\begin{eqnarray}
\frac{1}{2} \dfrac{d}{dt} ||u||^{2}_{2} + d_{1} ||\nabla u||^{2}_{2} + \int_{\Omega} u^{1+p} v dx +  \int_{\Omega} u^{3}dx  &=&  \int_{\Omega}m(x)u^{2} dx.    \nonumber \\
\end{eqnarray}

Using the estimate on $v$ from \eqref{eq:v1} we obtain

\begin{eqnarray}
\frac{1}{2} \dfrac{d}{dt} ||u||^{2}_{2} + d_{1} ||\nabla u||^{2}_{2} + C_{3} e^{-C_{2}t}\int_{\Omega} u^{1+p}dx +  \int_{\Omega} u^{3}dx  &\leq&  \int_{\Omega}m(x)u^{2} dx.    \nonumber \\
\end{eqnarray}

Here $C_{3} = C_{1} \min{v_{0}(x)} > 0$, then It follows that,

\begin{equation}
\frac{1}{2} \dfrac{d}{dt} ||u||^{2}_{2} + \min(d_{1}, C_{3}) e^{-C_{2}t} \left( ||\nabla u||^{2}_{2} + \int_{\Omega} u^{1+p}dx \right) +  \int_{\Omega} u^{3}dx  \leq  \int_{\Omega}m(x)u^{2} dx.   \nonumber \\
\end{equation}

 Thus we have that,

\begin{equation}
\frac{1}{2} \dfrac{d}{dt} ||u||^{2}_{2} + C_{4} e^{-C_{2}t} \left( ||\nabla u||^{2}_{2} + \int_{\Omega} u^{1+p}dx \right) +  \int_{\Omega} u^{3}dx  \leq ||m(x)||_{\infty}||u||^{2}_{2} dx.    \nonumber \\
\end{equation}

Our goal is to show that 

\begin{equation}
 \left( || u||^{2}_{2}  \right) ^{\alpha}  \leq  C_{4}\left( ||\nabla u||^{2}_{2} + \int_{\Omega} u^{1+p}dx \right) 
\end{equation}

where $0< \alpha < 1$, then we will have the finite time extinction of $u$ in analogy with the ODE

\begin{equation}
\frac{dy}{dt} = C_{5}y- C_{4}e^{-C_{2}t}y^{\alpha}, 0< \alpha < 1, C_{2}, C_{4}, C_{5} > 0.
\end{equation}

Now recall the Gagliardo-Nirenberg-Sobolev (GNS) inequality \cite{Sell2013},

	\begin{equation}
	||\phi||_{W^{k,p^{'}}(\Omega)} \leq C ||\phi||^{\theta}_{W^{m,q^{'}}(\Omega)} ||\phi||^{1 - \theta}_{L^{q}(\Omega)} 
	\end{equation}
for $\phi \in W^{m,q}(\Omega)$ provided $p^{'}, q^{'}, q \geq 1, 0 \leq \theta \leq 1$, and

\begin{equation}
k - \frac{n}{p^{'}} \leq \theta \left(  m - \frac{n}{q^{'}} \right) - (1 - \theta) \frac{n}{q}.
\end{equation}

Now consider exponents s.t.

\begin{equation}
W^{k,p^{'}}(\Omega) = L^{2}(\Omega), \ W^{m,q^{'}}(\Omega) = W^{1,2}(\Omega), \ L^{q}(\Omega) = L^{1+p}(\Omega)
\end{equation}

for $0 < p < 1$.

This yields

\begin{equation}
\label{eq:uest}
	||u|_{L^{2}(\Omega)} \leq C ||\phi||^{\theta}_{W^{1,2}(\Omega)} ||\phi||^{1 - \theta}_{L^{q}(\Omega)}, 
	\end{equation}
	
 as long as

%

	\begin{equation}
	\label{eq:1h}
	\frac{2-q}{2+q} \leq \theta \leq 1.
	\end{equation}

We raise both sides of \eqref{eq:uest} to the power of $l$	, $0<l<2$, to obtain

%
	
	\begin{equation}
	\left( \int_{\Omega} u^{2}dx \right)^{\frac{l}{2}} \leq C\left( \int_{\Omega} \nabla u^{2}dx \right)^{\frac{l \theta}{2}} \left( \int_{\Omega}  u^{q}dx \right)^{\frac{l (1-\theta)}{q}}.
	\end{equation}
	
	Using Young's inequality on the right hand side (for $a b \leq \frac{a^{r}}{r} + \frac{b^{m}}{m}$), with $r = \frac{2}{l \theta}, \ m= \frac{q}{l (1-\theta)}$, yields
	
	\begin{equation}
	\left( \int_{\Omega} u^{2}dx \right)^{\frac{l}{2}} \leq C \left( \int_{\Omega} \nabla u^{2}dx + \int_{\Omega}  u^{q}dx \right).
	\end{equation}
	
	We notice that given any $1<q<2$, it is always possible to choose $0<l<2$, s.t, $\frac{1}{r} + \frac{1}{m} = 1$,
	
	\begin{equation}
\frac{1}{r} + \frac{1}{m} = 	\frac{l \theta}{2}  + \frac{l (1-\theta)}{q} = 1,
	\end{equation}
	
	by choosing
	
	\begin{equation}
	\theta  = \frac{\frac{1}{l} - \frac{1}{q} }{\frac{1}{q} - \frac{1}{2} } = \frac{2(q-l)}{l(2-q)},
	\end{equation} 
	
	thus we need to choose $l$ s.t,
	
	\begin{equation}
	 \frac{2(q-l)}{l(2-q)} \geq \frac{2-q}{2+q}.
	\end{equation}

This enables the application of Young's inequality above, within the required restriction \eqref{eq:1h}, enforced by the GNS inequality.



Thus we have
\begin{equation}
\frac{1}{2} \dfrac{d}{dt} ||u||^{2}_{2} + C_{4} e^{-C_{2}t} \left(  ||u||^{2}_{2} \right)^{\frac{l}{2}}   \leq C_{5}||u||^{2}_{2} dx. \nonumber \\
\end{equation}
Let  $\alpha=\frac{l}{2} <1$ we have that $||u||^{2}_{2} \rightarrow 0$ as $t \rightarrow T^{*} < \infty$, for appropriately chosen initial data, in analogy with the ODE,

\begin{equation}
\frac{dy}{dt} = C_{5}y- C_{4}e^{-C_{2}t}y^{\alpha}, 0< \alpha < 1, C_{2}, C_{4}, C_{5} > 0.
\end{equation}

We set $y=g(t)e^{C_{5}t}$, to obtain

\begin{equation}
\label{eq:Ode}
  \frac{dg}{dt}=-C_{4}e^{-C_{6}t}(g(t))^{\alpha}, 0< \alpha < 1, C_{6}, C_{4}, C_{5} > 0.
\end{equation}

Solving eqn.(\ref{eq:Ode}) yields 
\begin{equation}
g(t)=\Bigg( \dfrac{(1-\alpha) C_4 e^{-C_6 t}}{C_6} + K \Bigg)^ \frac{1}{1-\alpha}, K~ \text{a~constant.}
\end{equation}

Here $K=(g_{0})^{(1-\alpha) } - \frac{((1-\alpha) )C_4}{C_6}$. Thus for initial data chosen s.t., $g(0) < \left(\frac{(1-\alpha)C_4}{C_6}\right)^{\frac{1}{1-\alpha}}$, then $g$ goes extinct at finite time $T^{*} = \ln \left(  \frac{C_{6}}{(1-\alpha)C_{4}-C_{6}(g_{0})^{1-\alpha } }\right)$, and so does $y(t)$. Thus we need to choose the initial data s.t. $||u_{0}||^{2}_{2} < \left(\frac{(1-\alpha)C_4}{C_6}\right)^{\frac{1}{1-\alpha}}$.
Since $L^{2}(\Omega)$ convergence implies uniform convergence on $\Omega$, which is closed and bounded, we see that for sufficiently chosen data $(u,v) \rightarrow (0,v^{*}(x))$ uniformly, and this occurs in finite time. However, if $p=1$, classical results \cite{Dockery1998}, would imply the same data would have converged to $(u^{*}(x),0)$.
This completes the proof.

\end{proof}

\newcommand{\overbar}[1]{\mkern 1.5mu\overline{\mkern-1.5mu#1\mkern-1.5mu}\mkern 1.5mu}

\subsection{The Weak Competition Case}

In the event that $0<b, c<1$ in \eqref{eq:Ge1pn}-\eqref{eq:Ge1bh}, we are in the weak competition case. Herein, if $d_{1} < d_{2}$, the slower diffuser could win or coexistence can occur. 
We define 
\newline
$\sum :=\Big \{(d_1,d_2) \in (0,\infty)\times (0,\infty) : (u^*,0) $ is stable\Big \} 
\newline
and recap a classical result \cite{He2016a},

\begin{theorem}\label{thm:classic}
Consider \eqref{eq:Ge1pn}-\eqref{eq:Ge1bh}, with $p=1$. Suppose that $0 < b \leq 1, c \in (c^*, 1)$ and m is non-constant. If $(d_1, d_2) \in \sum$, then $(u^*, 0)$ is globally stable among all non-negative and non-trivial initial conditions; if $(d_1, d_2) \notin \overbar{\sum}$ and $d_1 < d_2$, then (\ref{eq:Ge1pn}) admits a unique positive steady state which is globally stable.
\end{theorem}

Once we bring in FTE, that is $p<1$, numerical simulations illustrate interesting scenarios in the bifurcation plots in $(d_{1},d_{2})$ space. See Fig. \ref{fig:Diff_Diff} (a) for the classical result \cite{He2016a, He2016b} - however, when $p<1$, the bifurcation plot changes qualitatively, see Fig. \ref{fig:Diff_Diff} (b)-(c). We now define,
\newline
$\sum_{1} :=\Big \{(d_1,d_2) \in (0,\infty)\times (0,\infty) : (0,v^{*}) $ is stable\Big \}. 
\newline

This motivates the following conjecture,

\begin{conjecture}\label{con:c1}
Consider \eqref{eq:Ge1pn}-\eqref{eq:Ge1bh}. Suppose that $0 < b \leq 1, c \in (c^*, 1)$ and m is non-constant, then $\exists \ 0 <p<1$, s.t. If $(d_1, d_2) \in \sum$, then $(u^*, 0)$ is globally stable among all non-negative and non-trivial initial conditions; If $(d_1, d_2) \in \sum_{1}$, then $(0, v^{*})$ is globally stable among all non-negative and non-trivial initial conditions;
 if $(d_1, d_2) \notin \overbar{\sum \cup \sum_{1}}$ and $d_1 < d_2$, then (\ref{eq:Ge1pn}) admits a unique positive steady state which is globally stable.
\end{conjecture}

We also conjecture,

\begin{conjecture}\label{con:c2}
Consider \eqref{eq:Ge1pn}-\eqref{eq:Ge1bh}. Suppose that $0 < b \leq 1, c \in (c^*, 1)$ and m is non-constant, then for certain initial data, and any $\epsilon >0$, $\exists \ 0 <1-\epsilon < p <1$, s.t. If $(d_1, d_2) \in \sum$, then $(u^*, 0)$ is globally stable among all non-negative and non-trivial initial conditions; If $(d_1, d_2) \in \sum_{1}$, then $(0, v^{*})$ is globally stable among all non-negative and non-trivial initial conditions;
 if $(d_1, d_2) \notin \overbar{\sum \cup \sum_{1}}$ and $d_1 < d_2$, then (\ref{eq:Ge1pn}) admits a unique positive steady state which is globally stable.
\end{conjecture}

\begin{figure}[hbt!] 
\begin{center}
\subfigure[p=1]{
    \includegraphics[scale=.203]{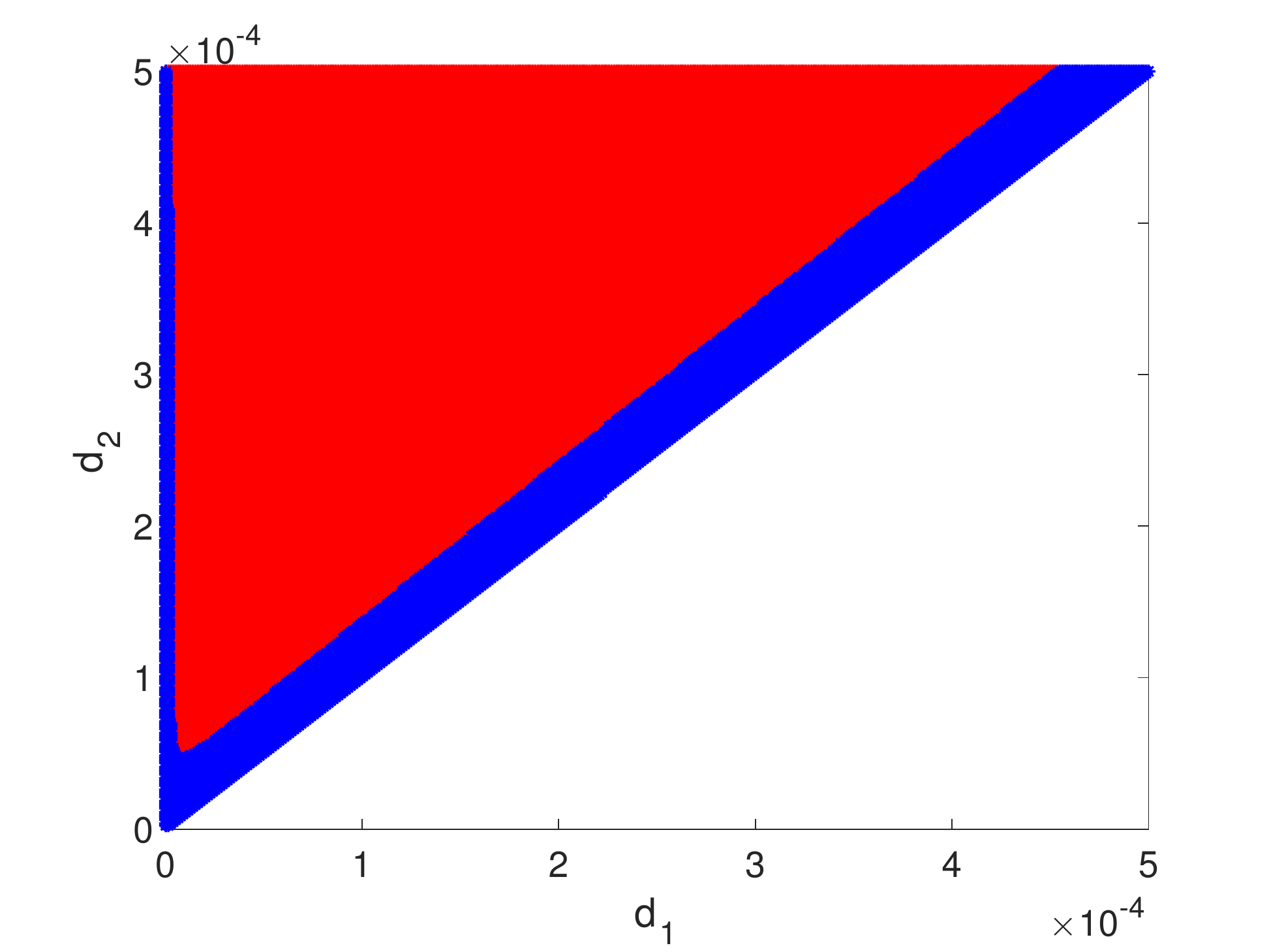}}
\subfigure[p=0.999]{    
    \includegraphics[scale=.203]{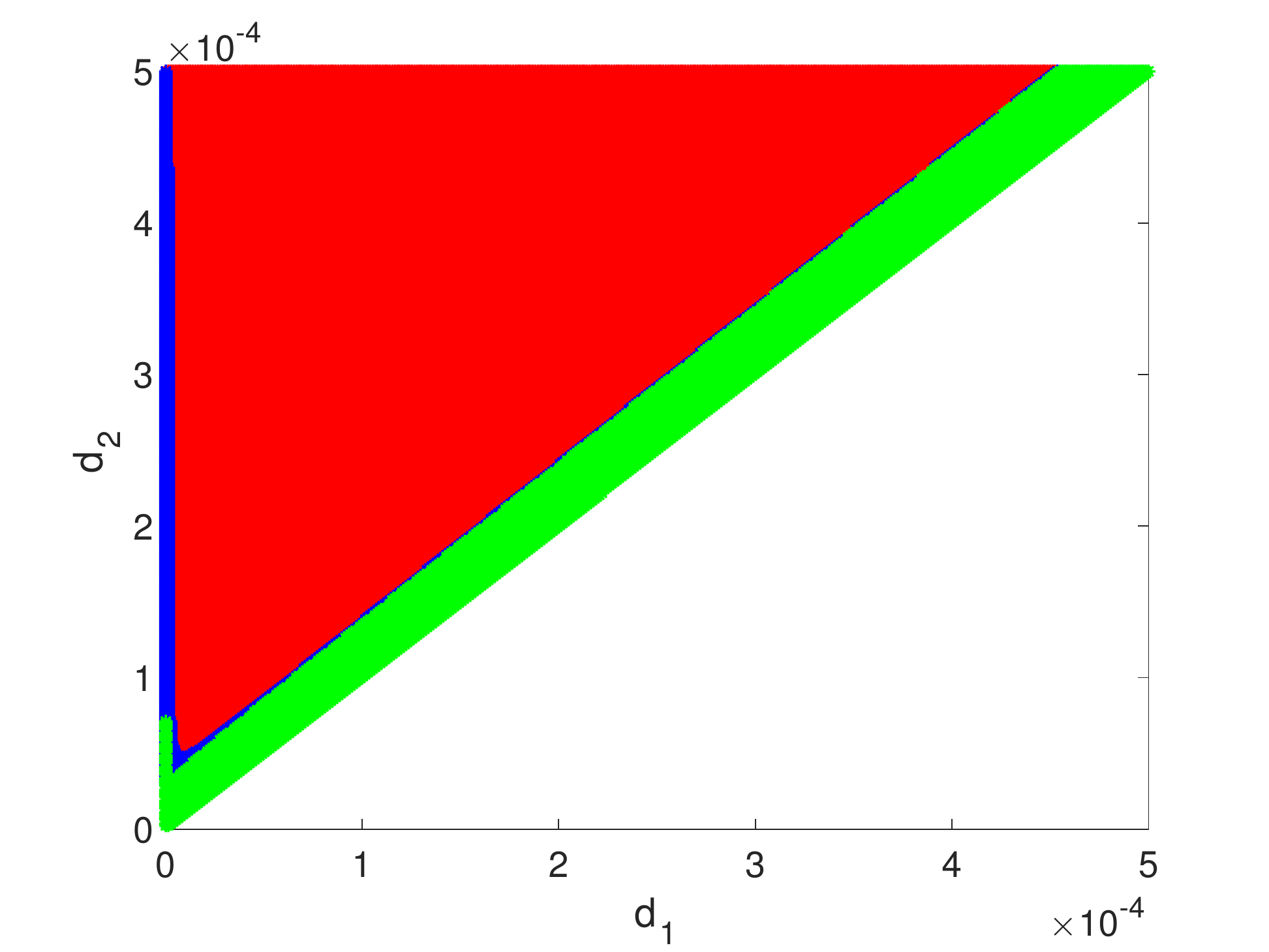}}
\subfigure[0.7]{
    \includegraphics[scale=.203]{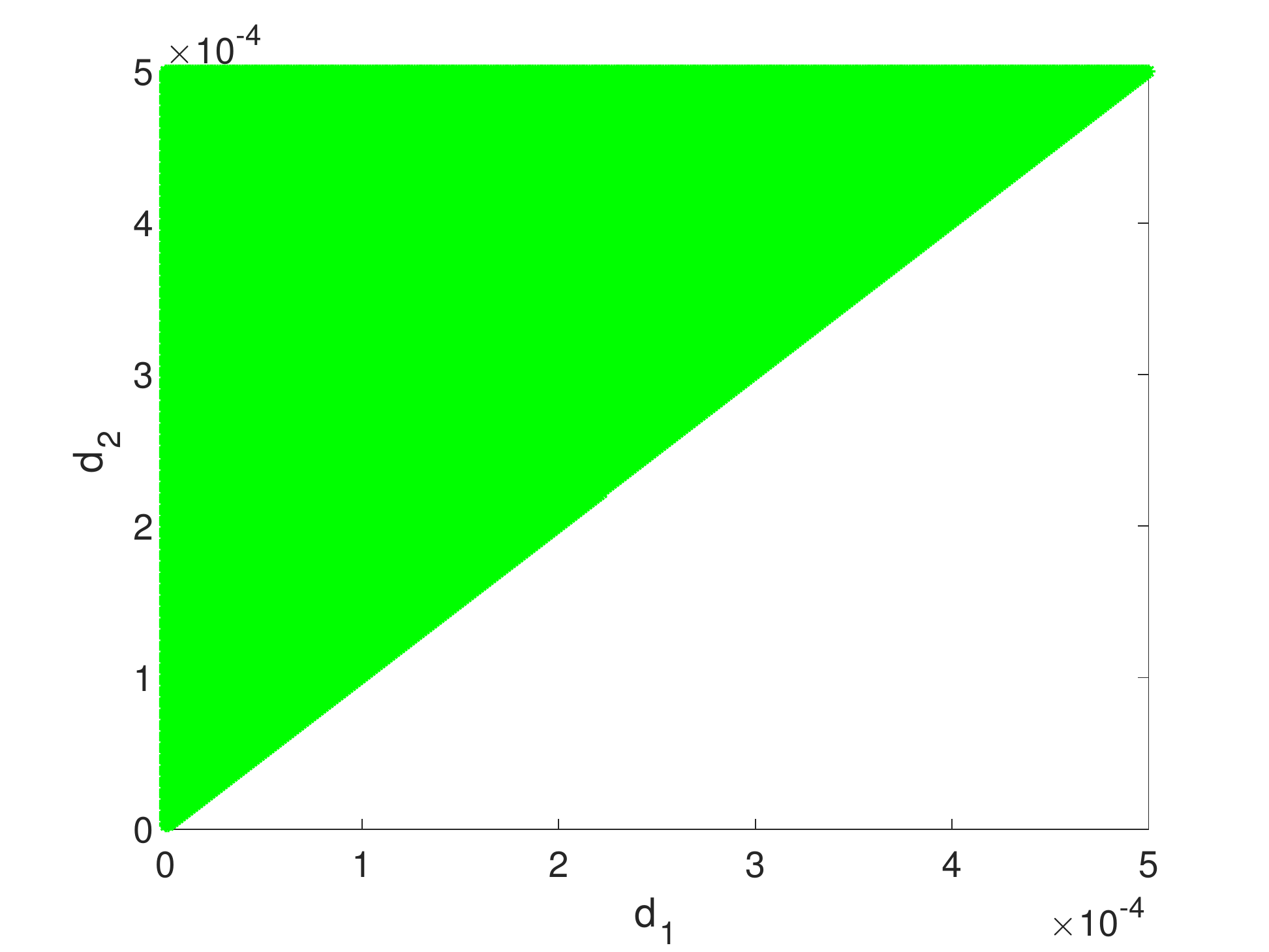}}
\end{center}
\caption{Plots of $d_1$ vs $d_2$ for \eqref{eq:Ge1pn}-\eqref{eq:Ge1bh}.  We choose $\Omega = [0, 1] $ and $m(x)=x(1-x)$. We use the following parameters: $b=c=0.999$.  The  red region shows that $u$ prevails, ie  ($u^*,0)$, blue region shows coexistence $(u^*,v^*)$ and green region shows $v$ prevailing, ie $(0,v^*)$. The classical results in theorem \ref{thm:classic} is seen in (a) and FTE results in (b) and (c) respectively.}
\label{fig:Diff_Diff}
\end{figure}

\begin{figure}[hbt!] 
\begin{center}
\subfigure[]{
    \includegraphics[scale=.31]{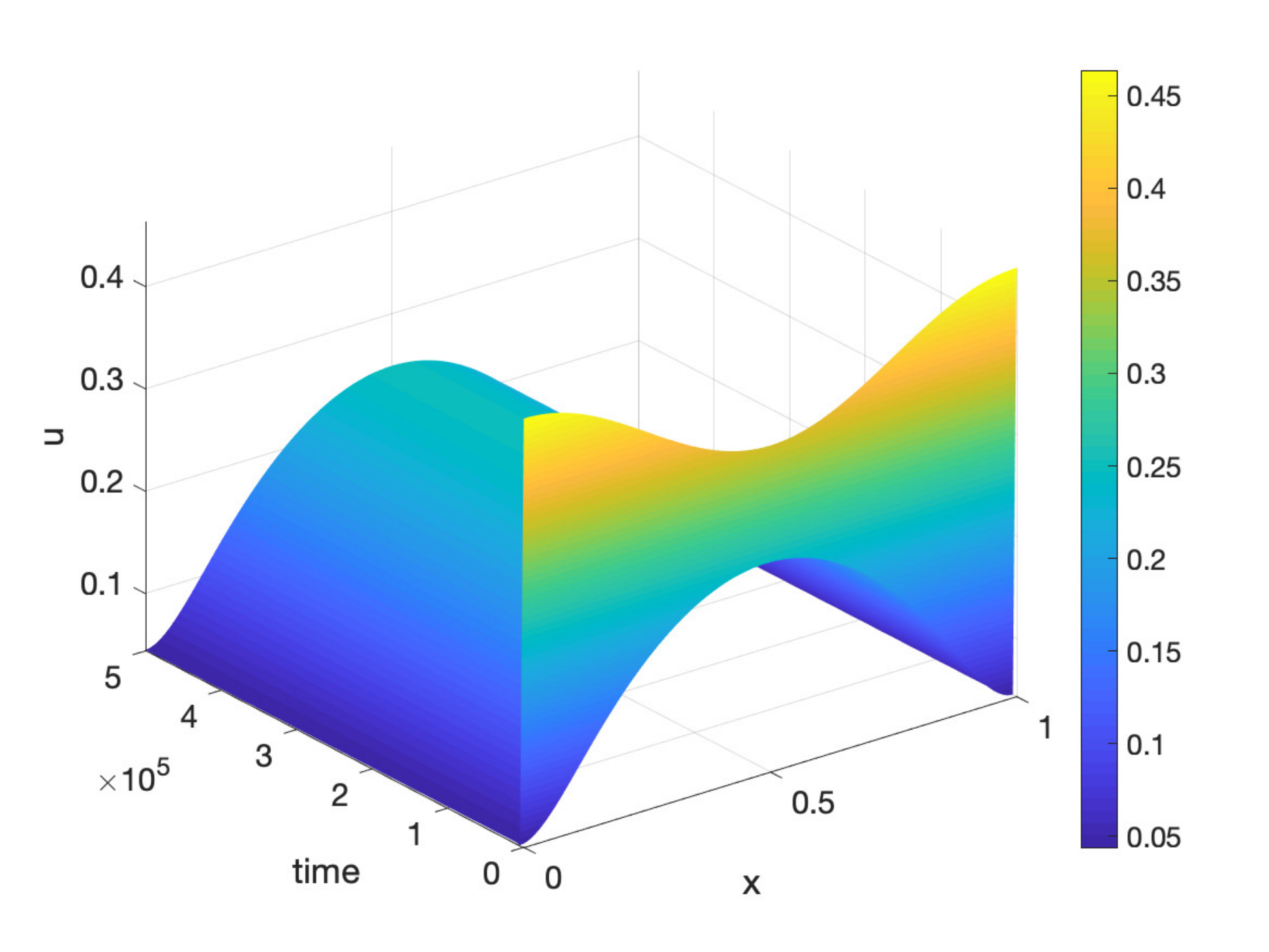}}
\subfigure[]{    
    \includegraphics[scale=.31]{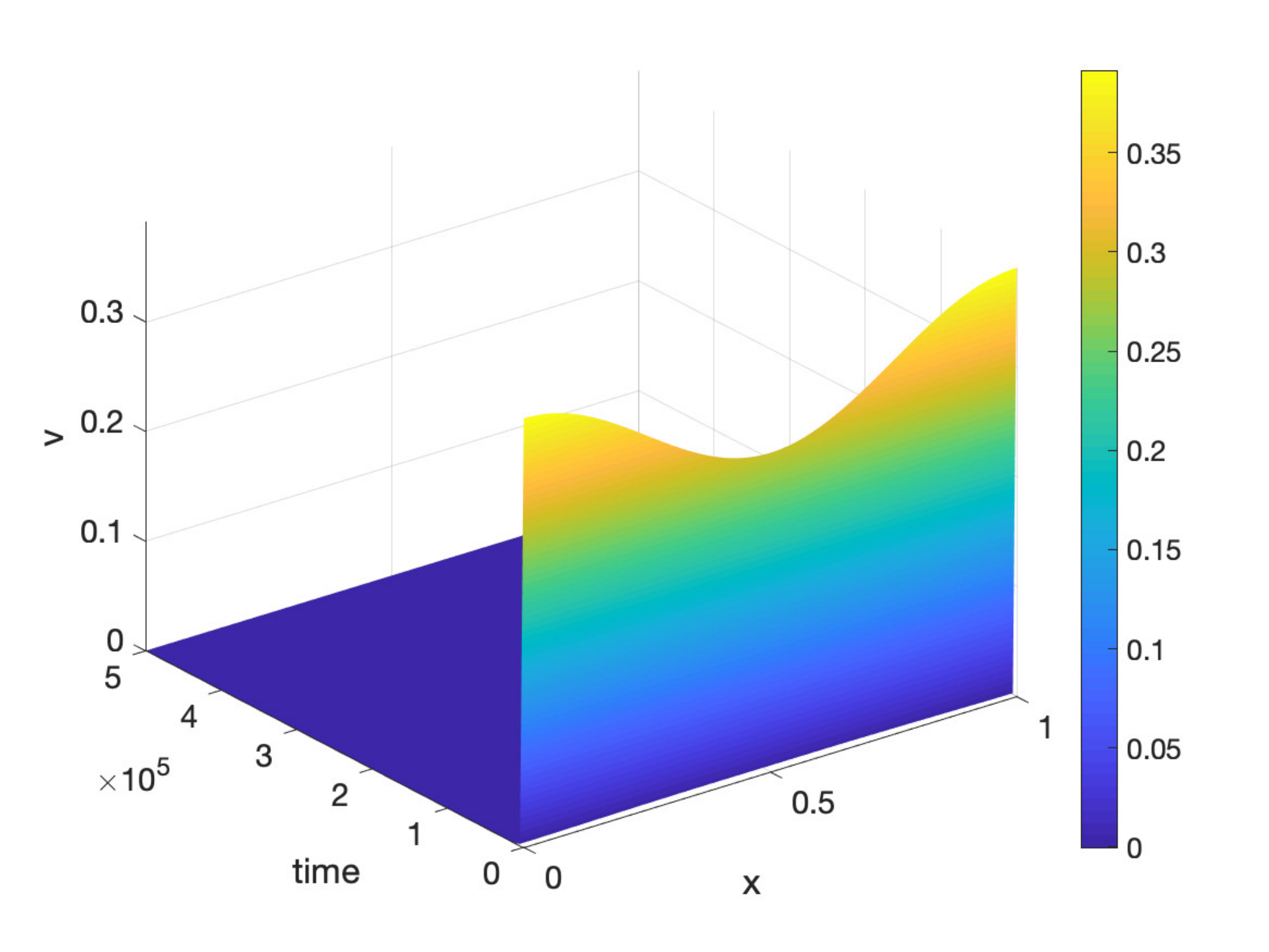}}
\subfigure[]{
    \includegraphics[scale=.31]{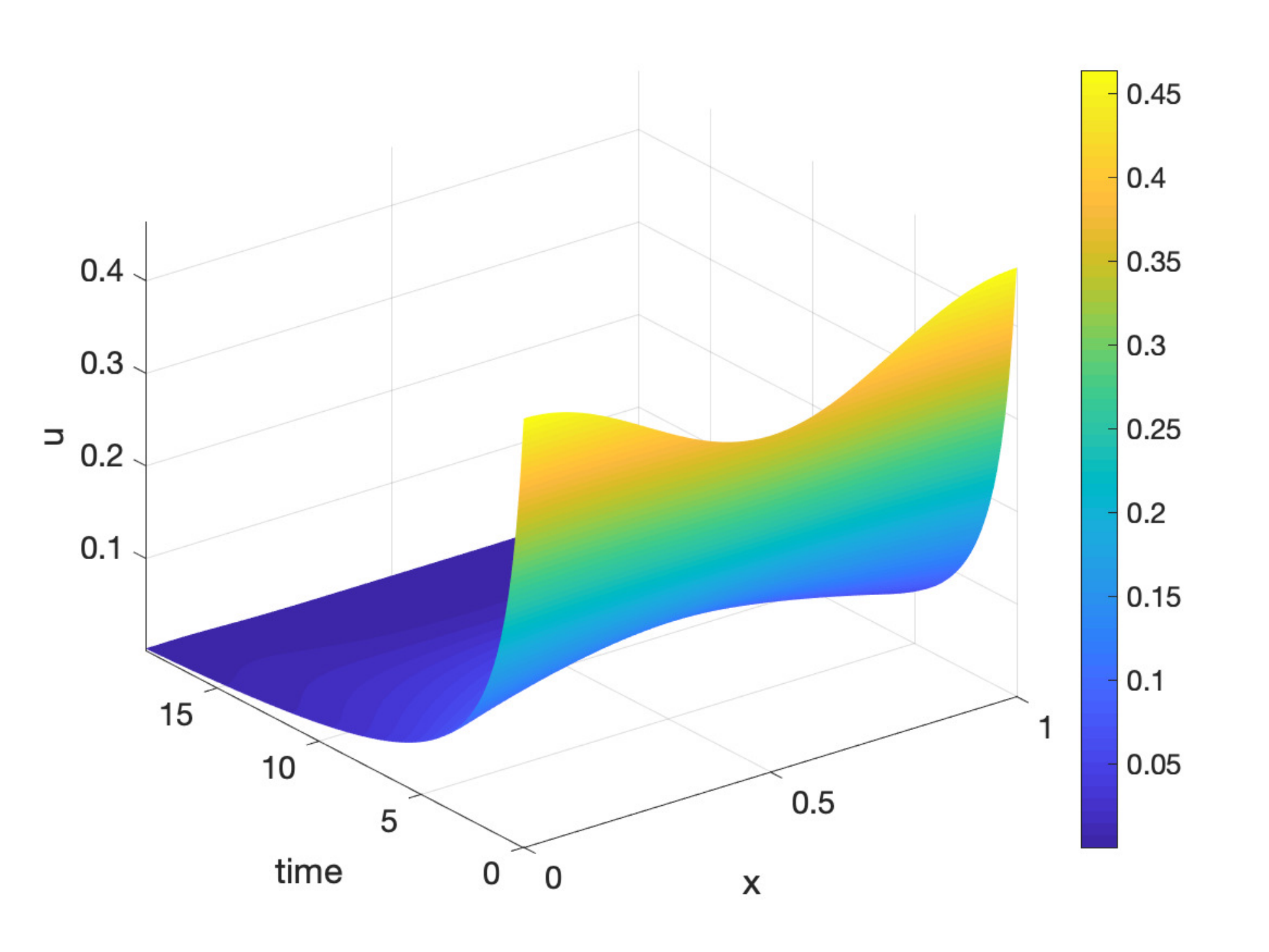}}
\subfigure[]{    
    \includegraphics[scale=.31]{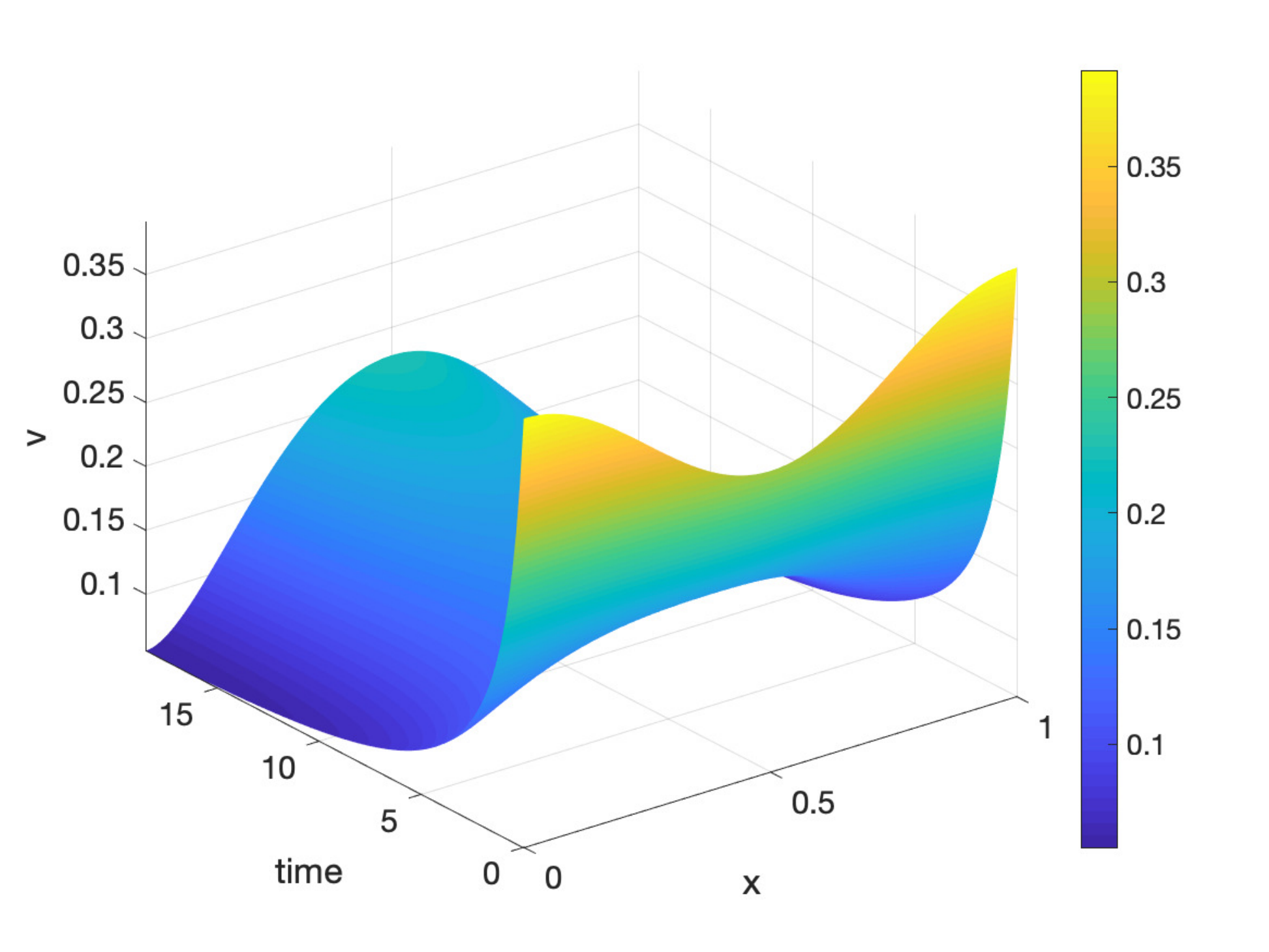}}    
\end{center}
\caption{Simulation showing solutions converging to $(u^*,0)$ in (a) and (b) for $p=1$. (c) and (d) show solutions converging to $(0,v^*)$ in finite time for $p=0.7$. The parameters used for \eqref{eq:Ge1pn}-\eqref{eq:Ge1bh} are $b=c=0.999, d_1=0.00012425 ~\text{and}~ d_2=0.00033167$. We choose $\Omega=[0,1] $ and $m(x)=x(1-x)$. The slower diffuser wins in (a) for $p=1$ and loses in (c) for $p<1$ in a weak competition case.}
\label{fig:newfte1}
\end{figure}

\section{Self Regulating or External Mechanisms of Control}
We consider the case where some proportion of the weaker competitor is harvested by an external controller or self regulates its population by an action such as cannibalism \cite{L20}. We ask if this ``strategy" might make it possible for stabilization of weaker population. We choose parametric restrictions according to the extinction case.
 Let $v= d v + e v$, here $d+e=1$, and $e$ is the proportion of the population that will possess the FTE dynamic. If $e=0, d=1$, we are in the competitive exclusion case \eqref{extinctioncondition2}. This leads us to the model,
\begin{equation}
\label{eq:ES1}
\left\{ \begin{array}{ll}
\dfrac{du }{dt} &~ = a_{1}u- b_{1}u^{2} - c_{1}auv ,\\[2ex]
\dfrac{dv }{dt} &~ =  a_{2}v - b_{2}v^{2} - c_{2} d uv- c_{2} e v^{q}.
\end{array}\right.
\end{equation}
We see that even in this setting $v$ can avoid competitive exclusion and persist, so coexist with the stronger competitor $u$. 
\begin{figure}[!htb]
\begin{center}
    \includegraphics[scale=.6]{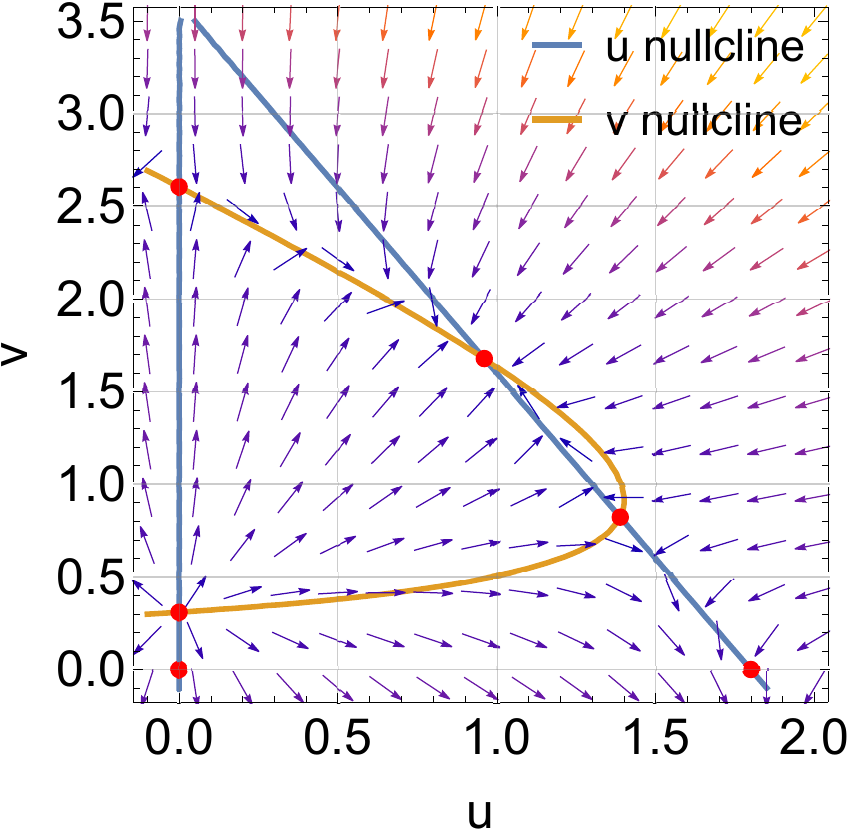}
\end{center}
 \caption{ External mechanisms of control \eqref{eq:ES1}: Extinction case \eqref{extinctioncondition2} for $a_{1}=1.8,~a_{2}=3,~b_{1}=1,~b_{2}=1,~c_{1}=0.5,~c_{2}=1.7,~a=1,~d=0.45,~e=0.55,~q=0.1$.}
      \label{fig:Extinction-Suicide1}
\end{figure}

\section{Discussions and Conclusions}\label{section:Dissussion_conclusion}

The current manuscript considers the two species ODE and PDE Lotka-Volterra competition model, where one competitor possesses the dynamic of FTE. As mentioned this is of immense interest currently to mathematicians and ecologists alike, in particular there is effort to understand in what capacity species will ``optimise" \cite{Bai2016, Caubet, Ding2010, Ni2020}. We see that bringing in FTE can change (albeit counterintuitively) certain classical ecological scenarios. Most notably, in the ODE case, we see that the weaker competitor can avoid competitive exclusion with the FTE dynamic - this is counterintuitive as it posits, that speeding up its extinction, enables it to turn the tables on a stronger competitor and coexist. This bodes interesting consequences for bio-control applications \cite{L20}, as well as motivates the use of such mechanisms in insect resistance management strategies, where two competing biotypes of a pest species are preferred to coexist \cite{O18} - our results could be used to develop tactics in these directions. Note, from an applied point of view, the FTE can be engineered by self regulating mechanisms or external control as well, via \eqref{eq:ES1}, thus a future direction could be a detailed investigation of such models. Also interesting, would be considering models where the stronger competitor counters the FTE dynamic in the weaker competitor with its own FTE dynamic. 

In the PDE case, Fig. \ref{fig:Diff_Diff}, is immensely interesting both from a mathematical and evolutionary point of view. Mathematically we aim to focus on a proof of conjecture \ref{con:c1}. From an evolutionary point of view, what we see is that the FTE dynamic, takes away some of the competitive advantage the slower diffuser has, in that if $d_{1} < d_{2}$, but close to $d_{2}$, the faster diffuser may win, and thus be selected for. This is the ``green" band seen in Fig. \ref{fig:Diff_Diff} (b). However, as $p$ is decreased, the advantage of slow diffusion, is taken away further and only $(0,v^{*})$ is observed, Fig. \ref{fig:Diff_Diff} (c). Conjecture \ref{con:c2} hypothesizes, that this taking away of competitive advantage, can be done for $p$ as close to 1 as possible - and in this setting we will see a plot qualitatively similar to Fig. \ref{fig:Diff_Diff} (b). Proving this would make for interesting future work. Another worthwhile future direction will be an extensive numerical simulation across a broader parameter range, to investigate how these dynamics might be effected. Also, such results may/may not hold in time varying environments \cite{Hutson2001}, this is also worthy of future investigations in light of the FTE dynamic.
\section{Appendix}
\label{a1}
\subsection{Analytic Guidelines}
We present analytic guidelines in this section to analyze the model \eqref{eq:Ge1} and to investigate its equilibria. Consider the solutions to the steady state equations:
\begin{align}
u \left[a_{1}  -b_{1} u - c_{1}u^{p-1}v\right]   &= 0,  \label{equilibrium1} \\
v\left[a_{2} -b_{2} v - c_{2} uv^{q-1}\right]  &= 0.   \label{equilibrium2}
\end{align}
The above equations, \eqref{equilibrium1} and \eqref{equilibrium2},  have four types of non-negative equilibria:

 \begin{enumerate}[label=(\roman*)]
 \item  $E_0 (0,0)$;
 \item  $E_1(a_1/b_1,0)$;
 \item  $E_2(0,a_{2}/b_{2})$;
 \item  $E_3(u^*, v^*)$; for $q=1$, we have

\begin{align}
u^*  &=\dfrac{1}{c_{2}}\left[a_{2}-b_{2} v^*\right],   \label{equilibrium3} \\
v^*  &=\dfrac{1}{c_1}\left[ a_{1} (u^*)^{1-p} -b_{1} (u^*)^{2-p}\right]  \label{equilibrium4}
\end{align}

and for $p=1$, we have

\begin{align}
v^*  &=\dfrac{1}{c_{1}}\left[a_{1}-b_{1} u^*\right],   \label{equilibrium5} \\
u^*  &=\dfrac{1}{c_2}\left[ a_{2} (v^*)^{1-q} -b_{2} (v^*)^{2-q}\right].  \label{equilibrium6}
\end{align}

The possible existence of  a unique interior or multiple equilibria are shown in Figs.~\ref{fig:Extinction-FTE1}, ~\ref{fig:Weak-FTE1}, ~\ref{fig:Weak-FTE2} and ~\ref{fig:Strong-FTE1}.
 \end{enumerate}
Now we discuss the local stability of an interior equilibrium point. The Jacobian  matrix $\bf{J}$ of the model \eqref{eq:Ge1} evaluated at any of the possible interior equilibria
$E_3(u^*, v^*)$ is
\begin{align*}
 \bf{J}=  \begin{bmatrix}
      a_1 - 2b_1 u^*-p c_1 {u^*}^{p-1} v^* &   - c_1 {u^*}^{p}\\
      - c_2 {v^*}^{q}  &   a_2 - 2b_2 v^*-q c_2 u^* {v^*}^{q-1}
     \end{bmatrix}.
\end{align*}

The characteristic equation corresponding to ${\bf J}$ is given by
 $$\lambda^2 - \operatorname{tr} \,({\bf{J}}) \lambda + \det \,({\bf{J}}) =0,$$
where
\begin{eqnarray*}
\operatorname{tr} \,({\bf{J}}) 
& = & a_1+a_2 - 2b_1 u^*- 2b_2 v^*-p c_1 {u^*}^{p-1} v^* -q c_2 u^* {v^*}^{q-1},
\end{eqnarray*}
and
\begin{eqnarray*}
\det \,({\bf{J}})
 &=& \left(  a_1 - 2b_1 u^*-p c_1 {u^*}^{p-1} v^* \right) 
  \left(a_2 - 2b_2 v^*-q c_2 u^* {v^*}^{q-1}\right) - c_1 c_2 {u^*}^{p} {v^*}^{q}.
\end{eqnarray*}
Here, $\operatorname{tr} \,({\bf{J}})$ and $\det \,({\bf{J}}) $ represent the trace and determinant of the Jacobian matrix.  Hence the stability of $E_3(u^*, v^*)$ is determined by the sign of $\det \,({\bf{J}}) $ and $\operatorname{tr} \,({\bf{J}})$.

The above results are summarized in the following theorem,

\begin{theorem}\label{localstability}
The interior equilibrium $E_3(u^*, v^*)$ of model \eqref{eq:Ge1} is locally asymptotically stable if  $\operatorname{tr} \,({\bf{J}}) < 0$ and $\det \,({\bf{J}})  > 0$ by Routh-Hurwitz stability criteria.
\end{theorem}

\begin{remark}\label{remark:saddle}
If  $\operatorname{tr} \,({\bf{J}})\geq0$ or $<0$  and  $\det \,({\bf{J}})  < 0$, then the roots of model \eqref{eq:Ge1} are both real numbers with opposite sign. Hence $E_3(u^*, v^*)$ is a saddle.
\end{remark}

\begin{example}\label{example1}
We provide justification for the above results by using the following set of parameter values:  $a_{1}=1.8,~a_{2}=3,~b_{1}=1,~b_{2}=1,~c_{1}=0.5,~c_{2}=1.8,~p=1,~q=0.3$. The interior equilibria $E_3^{1}(1.1323,1.3354)$ and $E_3^2(0.5788,2.4424)$ emerge with the Jacobians ${\bf{J^{*}}}$ and ${\bf{J^{**}}}$ respectively, where
 \begin{align*}
 \bf{J^*}=  \begin{bmatrix}
     -1.1323 &   -0.5662\\
      -1.9632 &   -0.1702
     \end{bmatrix}
    \quad \text{and} \quad
     \bf{J^{**}}=  \begin{bmatrix}
     -0.5788 &   -0.2894\\
      -2.3530 &   -2.0521
     \end{bmatrix}. 
\end{align*}
The $\operatorname{tr} \,({\bf{J^*}})=-1.3025<0$ and  $\det \,({\bf{J^*}})=-0.9188< 0$, thus  conditions for the saddle are satisfied.
Also, $\operatorname{tr} \,({\bf{J^{**}}})=-2.6309<0$ and  $\det \,({\bf{J^{**}}})=0.5068> 0$, thus the conditions for local stability are satisfied. We provide simulation in Fig. \ref{fig:Extinction-FTE1}(b) to validate. 
\end{example}

\section*{Conflict of Interest}
The authors declare there is no conflict of interest in this paper.

\section*{Acknowledgements}
 RP and ET would like to acknowledge valuable partial support from the National Science Foundation via DMS 1839993. KAF is partially supported by  Samford Faculty Development Grant (FUND 243084).

\appendix

\bibliographystyle{mdpi}

\renewcommand\bibname{References}


\begin{thebibliography}{999}




\bibitem{BS19}	
\newblock Beroual, N., Sari, T. (2020). \emph{A predator-prey system with Holling-type functional response,} Proceedings of the American Mathematical Society, DOI: 10.1090/proc/15166.

\bibitem{KRM20}	
\newblock Antwi-Fordjour, K., Parshad, R. D., Beauregard, M. A. (2020). \emph{Dynamics of a predator-prey model with generalized functional response and mutual interference,} Math. Biosci., 326 108407, DOI: 10.1016/j.mbs.2020.108407.

%
\bibitem{Murray93}	
\newblock Murray, J.D. (1993). \emph{Mathematical biology,} Springer, New York.

%


\bibitem{L12}
\newblock McKenzie, H. W., Merrill, E. H., Spiteri, R. J., $\&$ Lewis, M. A. (2012).
\newblock  \emph{How linear features alter predator movement and the functional response,} Interface focus, 2(2), 205-216.

\bibitem{M04}
\newblock Mols, C. M., van Oers, K., Witjes, L. M., Lessells, C. M., Drent, P. J.,  $\&$ Visser, M. E. (2004). \newblock \emph{Central assumptions of predator–prey models fail in a semi–natural experimental system,} \newblock Proceedings of the Royal Society of London B: Biological Sciences, 271(Suppl 3), S85-S87.

\bibitem{O18} O’Neal, M. E., Varenhorst, A. J., $\&$ Kaiser, M. C. (2018). 
\newblock \emph{Rapid evolution to host plant resistance by an invasive herbivore: soybean aphid (Aphis glycines) virulence in North America to aphid resistant cultivars,} Current opinion in insect science, 26, 1-7.


\bibitem{M20}
\newblock Mehta, P., McAuley, D. F., Brown, M., Sanchez, E., Tattersall, R. S., $\&$ Manson, J. J. (2020). \emph{COVID-19: consider cytokine storm syndromes and immunosuppression,} The Lancet, 395(10229), 1033-1034.


\bibitem{R05}
\newblock Ruxton, G. D. (2005).
 \newblock \emph{Increasing search rate over time may cause a slower than expected increase in prey encounter rate with increasing prey density,}
 \newblock Biology letters, 1(2), 133-135.

\bibitem{I08} Christos C. Ioannou, Graeme D. Ruxton, Jens Krause, \emph{Search rate, attack probability, and the relationship between prey density and prey encounter rate,} Behavioral Ecology, Volume 19, Issue 4, July-August 2008, Pages 842–846, https://doi.org/10.1093/beheco/arn038


 \bibitem{D97}
Dwyer, G., Elkinton, J. S., $\&$ Buonaccorsi, J. P. (1997). \emph{Host heterogeneity in susceptibility and disease dynamics: tests of a mathematical model,} The American Naturalist, 150(6), 685-707. 
 
  
  \bibitem{R11}  Ragsdale, D. W., Landis, D. A., Brodeur, J., Heimpel, G. E., $\&$ Desneux, N. (2011). 
  \newblock \emph{Ecology and management of the soybean aphid in North America,} 
  Annual review of entomology, 56, 375-399.




\bibitem{FCGT18}
Farrell, A. P., Collins, J. P., Greer, A. L., $\&$ Thieme, H. R. (2018). 
 \newblock \emph{Do fatal infectious diseases eradicate host species?},
 Journal of mathematical biology, 77(6-7), 2103-2164.

\bibitem{FT18}
Farrell, A. P., Collins, J. P., Greer, A. L., $\&$ Thieme, H. R. (2018). 
 \newblock \emph{Times from infection to disease-induced death and their influence on final population sizes after epidemic outbreaks},
  Bulletin of mathematical biology, 80(7), 1937-1961.
  
  \bibitem{F02}
  Fenton, A., Fairbairn, J. P., Norman, R., $\&$ Hudson, P. J. (2002). 
  \newblock \emph{Parasite transmission: reconciling theory and reality},
   Journal of Animal Ecology, 71(5), 893-905.
   
   \bibitem{BS02}  Bedjaoui, N., $\&$ Souplet, P. (2002). \emph{Critical blowup exponents for a system of reaction-diffusion equations with absorption,} Zeitschrift für angewandte Mathematik und Physik ZAMP, 53(2), 197-210.
   
    \bibitem{KF18} Fellner, K., Latos, E., $\&$ Tang, B. Q. (2018, May). \emph{Well-posedness and exponential equilibration of a volume-surface reaction–diffusion system with nonlinear boundary coupling,} In Annales de l'Institut Henri Poincaré C, Analyse non linéaire (Vol. 35, No. 3, pp. 643-673). Elsevier Masson.
    
   \bibitem{BV03} Bidaut‐Véron, M. F., García‐Huidobro, M., \& Yarur, C. (2003). \emph{On a semilinear parabolic system of reaction–diffusion with absorption,} Asymptotic Analysis, 36(3, 4), 241-283.    
    
    
    
    

  
  \bibitem{G08}
  Greer, A. L., Briggs, C. J., $\&$ Collins, J. P. (2008). 
  \newblock \emph{ Testing a key assumption of host‐pathogen theory: Density and disease transmission}, Oikos, 117(11), 1667-1673.



\bibitem{B12}
Braza, P. A. (2012).
 \newblock \emph{Predator–prey dynamics with square root functional responses}
 \emph Nonlinear Analysis: Real World Applications, 13(4), 1837-1843.


\bibitem{P19}
Parshad, R. D., Wickramsooriya, S., $\&$ Bailey, S. (2019).
\newblock \emph{A remark on “Biological control through provision of additional food to predators: A theoretical study”[Theor. Popul. Biol. 72 (2007) 111–120].}
 \emph  Theoretical Population Biology.






\bibitem{S97}
Sugie, J., Kohno, R., and Miyazaki, R. (1997).
\emph{On a predator-prey system of Holling type},
\emph Proceedings of the American Mathematical Society, 125(7), 2041-2050.

\bibitem{S99}
Sugie, J. and Katayama, M. (1999). \emph{Global asymptotic stability of a predator–prey system of Holling type}, Nonlinear Anal-Theor., Vol. 38, Iss. 1, 105-121.


   
\bibitem{Bai2016}  Bai, X.,  He X.,  Li F. (2016). 
\newblock\emph{An optimization problem and its application in population dynamics},
\newblock Proc. Amer. Math. Soc. 144 , 2161-2170.

\bibitem{Brown1980}  Brown, P.N. (1980). 
\newblock\emph{Decay to uniform states in ecological interactions}, 
\newblock SIAM J. Appl. Math. 38, 22-37.

\bibitem{Caubet}  Caubet, F., Deheuvels, T., Privat, Y. (2017). 
\newblock\emph{Optimal location of resources for biased movement of species: The 1D Case},
\newblock SIAM J. Appl. Math. 77, 1876-1903.
   
\bibitem{Cantrell2003}   Cantrell, R. S., Cosner, C. (2003).
\newblock\emph{Spatial Ecology via reaction-diffusion Equations}, 
\newblock Series in Mathematical and Computational Biology, John Wiley and Sons, Chichester, UK.   
   
\bibitem{Cantrell2004}  Cantrell, R.S., Cosner, C., Lou, Y.  (2004).
\newblock\emph{ Multiple reversals of competitive dominance in ecological reserves via external habitat degradation}, 
\newblock J. Dyn. Diff. Eqs. 16, 973-1010.   
   
\bibitem{Chen2020}Chen, S.S., Shi, J.P. (2020).
\newblock\emph{Global dynamics of the diffusive Lotka-Volterra competition model with stage structure},
\newblock Calc. Var. Partial Differential Equations 59:33.   
   
   \bibitem{CC18}  Cantrell, R. S., Cosner, C., $\&$ Yu, X. (2018). \emph{Dynamics of populations with individual variation in dispersal on bounded domains,} Journal of biological dynamics, 12(1), 288-317.
  
   
   
   
\bibitem{DeAngelis2016} DeAngelis, D., Ni, W.-M., Zhang,  B. (2016).
\newblock\emph{Dispersal and spatial heterogeneity: single species}, 
\newblock J. Math. Biol. 72, 239-254.   
  
  \bibitem{Ding2010}  Ding, W., Finotti, H., Lenhart,  S., Lou, Y., Ye, Q.  (2010).
\newblock\emph{Optimal control of growth coefficient on a steady-state population model},
\newblock Nonlinear Anal. Real World Appl. 11, 688-704.
  
\bibitem{Dockery1998}  Dockery, J., Hutson, V., Mischaikow, K., Pernarowski, M. (1998).
\newblock\emph{The evolution of slow dispersal rates: a reaction- diffusion model},
\newblock J. Math. Biol. 37, 61-83. 
  
\bibitem{Hastings}  Hastings, A. (1983).
\newblock\emph{Can spatial variation alone lead to selection for dispersal?}, 
\newblock Theor. Pop. Biol. 24, 244-251.
  
\bibitem{He2019}  He, X., Lam, K.-Y., Lou, Y., Ni, W.-M. (2019).
\newblock\emph{Dynamics of a consumer-resource reaction-diffusion model: Homo- geneous versus heterogeneous environments},
\newblock  J. Math Biol. 78, 1605-1636.  
  
  
   
  
  
\bibitem{He2013a}  He, X., Ni, W.-M. (2013).
\newblock\emph{The effects of diffusion and spatial variation in Lotka-Volterra competition-diffusion system I: Heterogeneity vs. homogeneity},
\newblock J. Differential Equations 254, 528-546.  
 
\bibitem{He2013b} He, X., Ni, W.-M.  (2013).
\newblock\emph{The effects of diffusion and spatial variation in Lotka-Volterra competition-diffusion system II: The general case}, 
\newblock J. Differential Equations 254, 4088-4108. 
  
\bibitem{He2016a} He, X., Ni, W.-M. (2016).
\newblock\emph{Global dynamics of the Lotka-Volterra competition-diffusion system: Diffusion and spatial heterogeneity I},
\newblock Comm. Pure. Appl. Math. 69, 981-1014.
  
\bibitem{He2016b} He, X., Ni, W.-M. (2016).
\newblock\emph{Global dynamics of the Lotka-Volterra competition-diffusion system with equal amount of total resources II},
\newblock Calc. Var. Partial Differential Equations 55 : 25. 
  
  \bibitem{Ninomiya1995} Hirokazu, Ninomiya (1995).
  \newblock\emph{Separatrices of competition-diffusion equations}.
  \newblock J. Math. Kyoto Univ.(JMKYAZ) 35-3, 539-567
  
  
  
  
\bibitem{Hutson2003} Hutson, V., Lou, Y., Mischaikow, K., Pol\'{a}\v{c}ik, P. (2003).
\newblock\emph{Competing species near the degenerate limit},
\newblock SIAM J. Math. Anal. 35, 453-491.   
   
\bibitem{Hutson2001} Hutson, V., Mischaikow, K., Pol\'{a}\v{c}ik, P. (2001).
\newblock\emph{The evolution of dispersal rates in a heterogeneous time-periodic environment},
\newblock  J. Math. Biol. 43, 501-533. 
   
   
\bibitem{Yao2008} Kao, C.Y.,  Lou, Y., Yanagida, E. (2008).
\newblock\emph{Principal eigenvalue for an elliptic problem with indefinite weight on cylindrical domains},
\newblock Math. Biosci. Eng. 5, 315-335.   
   
   
\bibitem{L20} Lyu, J., Schofield, P. J., Reaver, K. M., Beauregard, M., $\&$ Parshad, R. D. (2020). \emph{A comparison of the Trojan Y Chromosome strategy to harvesting models for eradication of nonnative species,} Natural Resource Modeling, 33(2), e12252.   
   
%
\bibitem{Lam2012} Lam, K.-Y., Ni, W.-M. (2012).
\newblock\emph{Uniqueness and complete dynamics of the Lotka-Volterra competition diffusion system}, 
\newblock SIAM J. Appl. Math. 72, 1695-1712.   
   
\bibitem{Li2019} Li, R., Lou, Y. (2019).
\newblock\emph{Some monotone properties for solutions to a reaction-diffusion model}, 
\newblock Discrete Contin. Dyn. Syst. Ser. B 24, 4445-4455.  
   
 \bibitem{Liang2012} Liang, S., Lou, Y. (2012).
\newblock\emph{On the dependence of the population size on the dispersal rate},
\newblock Discrete Contin. Dyn. Syst. Ser. B 17, 2771-2788.  
   

\bibitem{Lou2006a} Lou, Y. (2006).
\newblock\emph{On the effects of migration and spatial heterogeneity on single and multiple species}, 
\newblock J. Differential Equations 223, 400-426.    
 
\bibitem{Lou2006b} Lou, Y., Martinez, S., Pol\'{a}\v{c}ik, P. (2006).
\newblock\emph{Loops and branches of coexistence states in a Lotka-Volterra competition model}, 
\newblock J. Differential Equations 230, 720-742.
 
 \bibitem{Lou2008}Lou, Y. (2008).
 \newblock\emph{Some challenging mathematical problems in evolution of dispersal and population dynamics}, 
 \newblock Tutorials in mathematical biosciences. IV, 171-205, Lecture Notes in Math., 1922, Math. Biosci. Subser., Springer, Berlin.
 
\bibitem{Lou2017} Lou, Y.,  Wang, B. (2017). 
 \newblock\emph{Local dynamics of a diffusive predator-prey model in spatially heterogeneous environment},
 \newblock  J. Fixed Point Theory Appl. 19, 755-772.
 
 \bibitem{Masato1998} Masato, I., Tatsuya, M., Hirokazu, N. and Yanagida, E. (1998).
 \newblock\emph{ Diffusion-Induced Extinction of a Superior Species in a Competition System}.
 \newblock Japan J. Indust. Appl. Math. 15,233-252
 
 \bibitem{Mazari2020a} Mazari, I., Nadin, G., Privat, Y. (2020).
\newblock\emph{Optimal location of resources maximizing the total population size in
logistic models},
\newblock J. Math. Pure. Appl. 134, 1-35.
 
%
   
\bibitem{Nagahara2020}  Nagahara, K., Lou, Y., Yanagida, E. 
\newblock\emph{Maximizing the total population in a patchy environment},
\newblock  submitted, 2020.   

\bibitem{Nagahara2018}  Nagahara, K., Yanagida, E. (2018).
\newblock\emph{Maximization of the total population in a reaction-diffusion model with logistic growth},
\newblock Calc. Var. Partial Differential Equations 57 : 80.   

\bibitem{Ni2011} Ni, W.-M. (2011).
\newblock\emph{The Mathematics of Diffusion},
\newblock  CBMS Reg. Conf. Ser. Appl. Math. 82, SIAM, Philadelphia.



\bibitem{Ni2020} Ni, W., Shi, J., Wang, M. (2020).
\newblock\emph{Global stability of nonhomogeneous equilibrium solution for the diffusive Lotka-
Volterra competition model},
\newblock  Calc. Var. 59, 132 . 


 
\bibitem{Ni2012} Ni, Wei-Ming (2012).
\newblock\emph{Complete Dynamics in a Heterogeneous Competition-Diffusion System},
\newblock East China Normal University and University of Minnesota . 

\bibitem{Sell2013} Sell, G. R., You, Y. (2013).
\newblock\emph{Dynamics of evolutionary equations},
\newblock Vol. 143, Springer Science \& Business Media . 





\end{thebibliography}
\end{document}